\documentclass[11pt,reqno]{amsart}
\usepackage{graphicx} 
\usepackage{amsmath, mathtools, amsfonts, amssymb, amsthm}
\usepackage[margin=1in]{geometry}
\usepackage{color}
\usepackage{xcolor}
\usepackage{soul}
\usepackage{verbatim}
\usepackage[normalem]{ulem}
\usepackage{hyperref}

\def\F{\mathbb{F}}
\def\R{\mathbb{R}}
\def\C{\mathbb{C}}

\def\sgn{\mathop{\rm sgn}}

\newtheoremstyle{case}{}{}{}{}{}{:}{ }{}
\theoremstyle{case}

\theoremstyle{plain}

\newtheorem{theorem}{Theorem}[section]
\newtheorem{utheorem}{\textrm{\textbf{Theorem}}}

\newtheorem{corollary}[theorem]{Corollary}

\newtheorem{proposition}[theorem]{Proposition}
\newtheorem{lemma}[theorem]{Lemma}

\theoremstyle{definition}

\newtheorem{definition}[theorem]{Definition}
\newtheorem{remark}[theorem]{Remark}
\newtheorem{question}[theorem]{Question}

\numberwithin{equation}{section}

\begin{document}
\title{Positivity preservers over finite fields}

\author{Dominique Guillot}
\address[D.~Guillot]{University of Delaware, Newark, DE, USA}
\email{\tt dguillot@udel.edu}

\author{Himanshu Gupta}
\address[H.~Gupta]{University of Regina, Regina, SK, Canada}
\email{\tt himanshu.gupta@uregina.ca}

\author{Prateek Kumar Vishwakarma}
\address[P.~K.~Vishwakarma]{Indian Institute of Science, Bangalore, India}
\email{\tt prateekv@alum.iisc.ac.in}

\author{Chi Hoi Yip}
\address[C.~H.~Yip]{Georgia Institute of Technology, Atlanta, GA, USA}
\email{cyip30@gatech.edu}
\date{\today}



\begin{abstract}
We resolve an algebraic version of Schoenberg's celebrated theorem [\textit{Duke Math.~J.}, 1942] characterizing entrywise matrix transforms that preserve positive definiteness. Compared to the classical real and complex settings, we consider matrices with entries in a finite field and obtain a complete characterization of such preservers for matrices of a fixed dimension. When the dimension of the matrices is at least $3$, we prove that, surprisingly, the positivity preservers are precisely the positive multiples of the field's automorphisms. We also obtain characterizations of preservers in the significantly more challenging dimension $2$ case over a finite field with $q$ elements, unless $q \equiv 1 \pmod 4$ and $q$ is not a square. Our proofs build on several novel connections between positivity preservers and field automorphisms via the works of Weil, Carlitz, and Muzychuk-Kov\'acs, and via the structure of cliques in Paley graphs.
\end{abstract}

\maketitle

\section{Introduction and main Results}
Let $A = (a_{ij})$ be an $n \times n$ matrix with entries in a field $\F$ and let $f$ be a function defined on $\F$. The function naturally induces an entrywise transformation of $A$ via $f[A] := (f(a_{ij}))$. The study of such entrywise transforms that preserve various forms of matrix positivity has a rich and long history with important applications in many fields of mathematics such as distance geometry and Fourier analysis on groups -- see the surveys \cite{BGKP-survey-part-1, BGKP-survey-part-2} and the monograph \cite{khare2022matrix} for more details. Consider for example the set of $n \times n$ real symmetric or complex Hermitian matrices. By the well-known Schur product theorem \cite{Schur1911}, the entrywise product $A \circ B := (a_{ij} b_{ij})$ of two positive semidefinite matrices is positive semidefinite. As an immediate consequence of this surprising result, monomials $f(x) = x^n$ with $n \geq 1$, and more generally convergent power series $f(x) = \sum_{n=0}^\infty c_n x^n$ with real nonnegative coefficients $c_n \geq 0$ preserve positive semidefiniteness when applied entrywise to $n \times n$ real symmetric or complex Hermitian positive semidefinite matrices. An impressive converse of this result was obtained by Schoenberg \cite{Schoenberg-Duke42}, with various refinements by others collected over time \cite{Rudin-Duke59, BGKP-hankel, khare2022matrix}. 

\begin{theorem}[{\cite[Chapter 18]{khare2022matrix}}]\label{Tschoenberg}
Let $I= (-\rho,\rho)$, where $0 < \rho \leq \infty$.
Given a function $f : I \to \R$, the following are equivalent.
\begin{enumerate}
\item The function $f$ acts entrywise to preserve the set of positive semidefinite matrices of all dimensions with entries in~$I$.
\item The function $f$ is \emph{absolutely monotone}, that is,
$f(x) = \sum_{n=0}^\infty c_n x^n$ for all $x \in I$
with $c_n \geq 0$ for all $n$.
\end{enumerate}
Moreover, $f$ preserves the set of positive definite matrices of all dimensions with entries in~$I$ if and only if $f$ is absolutely monotone and non-constant.
\end{theorem}

Notice that in Schoenberg's result, the characterization applies to functions preserving positivity for matrices of arbitrarily large dimension. Obtaining a characterization of the entrywise preservers for matrices of a fixed dimension is a very natural endeavor, but a much harder problem that remains mostly unsolved. An interesting necessary condition given by Horn \cite{horn1969theory} shows that such preservers must have a certain degree of smoothness, with a number of non-negative derivatives. In \cite{BGKP-fixeddim}, seventy-four years after the publication of Schoenberg's result, Belton--Guillot--Khare--Putinar resolved the problem for polynomials of degree at most $N$ that preserve positivity on $N \times N$ matrices. They also provided the first known example of a non-absolutely monotone polynomial that preserves positivity in a fixed dimension. In \cite{Khare-Tao}, Khare and Tao characterized the sign patterns of the Maclaurin coefficients of positivity preservers in fixed dimension. They also considered sums of real powers, and uncovered exciting connections between positivity preservers and symmetric function theory. However, apart from this recent progress, the problem of determining entrywise preservers in fixed dimension remains mostly unresolved. We note that many other variants were previously explored, including problems involving: structured matrices \cite{BGKP-hankel,GKR-sparse, GKR-lowrank}, specific functions \cite{fitzgerald1977fractional, guillot2015complete, GKR-critG, guillot2012retaining, hiai2009monotonicity}, block actions \cite{guillot2015functions, vishwakarma2023positivity}, different notions of positivity \cite{BGKP-tn}, preserving inertia \cite{belton2023negativity}, and multivariable transforms \cite{belton2023negativity, fitzgerald1995functions}. 

Several authors have considered various preservers problems over finite fields (see e.g.~\cite{guterman2000some, li2001linear, orel2016preserver} and the references therein). However, to the authors' knowledge, all previous work on positivity preservers has focused on matrices with real or complex entries. In this paper, we consider matrices with entries in a finite field and describe the associated entrywise positivity preservers in the harder fixed-dimensional setting. As a consequence, we also obtain the positivity preservers for matrices of all dimensions, as in the setting of Schoenberg's theorem. Recall that in the real setting, a symmetric matrix in $M_n(\R)$ is positive definite if and only if all its leading principal minors are positive; see Proposition~\ref{Ppd} for other equivalent definitions. By analogy, we think of non-zero squares in a finite field $\F_q$ as positive elements in $\F_q$ and say that a symmetric matrix in $M_n(\F_q)$ is positive definite if all its leading principal minors are equal to the square of some non-zero element in $\F_q$. As shown in \cite{cooper2022positive}, this leads to a reasonable notion of positive definiteness for matrices with entries in finite fields. We therefore adopt the following definition.

\begin{definition}[Positive definite matrices over $\F_q$]\label{Dpd}
We say that a matrix $A \in M_n(\F_q)$ is {\it positive definite} if $A$ is symmetric and all its leading principal minors are non-zero squares in $\F_q$.
\end{definition}

\noindent Our goal is to classify entrywise preservers of positive definite matrices. 

\begin{definition}
    Given a matrix $A = (a_{ij}) \in M_n(\F_q)$ and a function $f: \F_q \to \F_q$, we denote by $f[A]$ the matrix obtained by applying $f$ to the entries of $A$: 
    \[
    f[A] := (f(a_{ij})).
    \]
   We say that $f$ {\it preserves positivity} (or is a {\it positivity preserver}) on $M_n(\F_q)$ if $f[A]$ is positive definite for all positive definite $A \in M_n(\F_q)$.
\end{definition}

We refer to Section~\ref{sec:pdm} for more background and motivation. Compared to previous work on $\R$ or $\C$ that uses analytic techniques to characterize preservers, the flavor of our work is considerably different and relies mostly on algebraic, combinatorial, and number-theoretic arguments. Surprisingly, our characterizations unearth new connections between functions preserving positivity, field automorphisms, and automorphisms of Paley graphs.

For each prime power $q$, we show that the positivity preservers on $M_n(\F_q)$, for a fixed $n\geq 3$, are precisely positive multiples of field automorphisms of $\F_q$. With a much more delicate analysis, we also give a complete classification of positivity preservers on $M_2(\F_q)$ for all prime powers $q$ other than those with $q \equiv 1 \pmod 4$ that are not a perfect square. Detailed statements of our main results including refinements are given in Theorems \ref{ThmA}, \ref{ThmB}, \ref{ThmC}, and \ref{ThmD} in Section \ref{SMainResults} below. 

\subsection{Main results}\label{SMainResults}

Let $p$ be a prime number and $k$ a positive integer. We denote the finite field with $q = p^k$ elements by $\F_q$. We let $\F_q^* := \F_q \setminus \{0\}$ denote the non-zero elements of the field. We say that an element $x \in \F_q$ is {\it positive} if $x = y^2$ for some $y \in \F_q^*$. In that case, we say $y$ is a square root of $x$. We denote the set of positive elements of $\F_q$ by $\F_q^+$, i.e., $\F_q^+ := \{x^2 : x \in \F_q^*\}$. Similarly, we denote the set of {\it negative} elements of $\F_q$ by $\F_q^-=\F_q^*\setminus \F_q^+$.
If $q$ is odd, then $|\F_q^+| =  |\F_q^-| = \frac{q-1}{2}$. When $q$ is odd, the {\it quadratic character} of $\F_q$ is the function $\eta: \F_q \to \{-1,0,1\}$ given by:  
\begin{equation}\label{EquadChar}
\eta(x) = x^{\frac{q-1}{2}} = \begin{cases}
    1 & \textrm{if } x \in \F_q^+ \\
    -1 & \textrm{if } x \in \F_q^- \\ 
    0 & \textrm{if } x = 0.
\end{cases}
\end{equation}
Finally, we denote by $M_n(\F_q)$ the set of $n \times n$ matrices with entries in $\F_q$, by $I_n$ the $n \times n$ identity matrix, and by ${\bf 0}_{m \times n}$ the $m \times n$ matrix whose entries are all $0$. 

In classifying the positivity preservers on $M_n(\F_q)$, a natural trichotomy arises. When $q$ is even, the Frobenius map $f(x) = x^2$ is an automorphism of $\F_q$ so that every non-zero element of $\F_q$ is a square. Characterizing the entrywise preservers in even characteristic thus reduces to characterizing the entrywise transformations that preserve non-singularity, a problem that is considerably different from the odd characteristic case. Our techniques in odd characteristics also differ depending on whether $-1$ is a square in $\F_q$. When $q$ is odd, it is well-known that $-1 \not\in \F_q^+$ if and only if $q \equiv 3 \pmod 4$. As a consequence, our work is organized into three parts: (1) the even characteristic case, (2) the $q \equiv 3 \pmod 4$ case where $-1 \not\in \F_q^+$, and (3) the $q \equiv 1 \pmod 4$ case where $-1 \in \F_q^+$. Our first main result addresses the even characteristic case. 

\begin{utheorem}\label{ThmA}
Let $q = 2^k$ for some positive integer $k$ and let $f: \F_q \to \F_q$. Then 
\begin{enumerate}
\item ($n=2$ case) The following are equivalent:
\begin{enumerate}
\item $f$ preserves positivity on $M_2(\F_q)$.
\item $f$ is a bijective monomial on $\F_q$, that is, there exist $c \in \F_q^*$ and $1 \leq n \leq q-1$ with $\gcd(n, q-1) = 1$ such that $f(x) = cx^n$ for all $x \in \F_q$.
\end{enumerate}
\item ($n \geq 3$ case) The following are equivalent:
\begin{enumerate} 
\item $f$ preserves positivity on $M_n(\F_q)$ for some $n \geq 3$.
\item $f$ preserves positivity on $M_n(\F_q)$ for all $n \geq 2$. 
\item $f$ is a non-zero multiple of a field automorphism of $\F_q$, i.e., there exist $c \in \F_q^*$ and $0 \leq \ell \leq k-1$ such that $f(x) = cx^{2^\ell}$ for all $x \in \F_q$.
\end{enumerate}
\end{enumerate}
\end{utheorem}

Our second main result addresses the case where $q \equiv 3 \pmod 4$. 

\begin{utheorem}\label{ThmB}
Let $q \equiv 3 \pmod 4$ and let $f: \F_q \to \F_q$. Then the following are equivalent: 
\begin{enumerate}
\item $f$ preserves positivity on $M_n(\F_q)$ for some $n \geq 2$. 
\item $f$ preserves positivity on $M_n(\F_q)$ for all $n \geq 2$. 
\item $f(0) = 0$ and $\eta(f(a)-f(b)) = \eta(a-b)$ for all $a, b \in \F_q$. 
\item $f$ is a positive multiple of a field automorphism of $\F_q$, i.e., there exist $c \in \F_q^+$ and $0 \leq \ell \leq k-1$ such that $f(x) = cx^{p^\ell}$ for all $x \in \F_q$. 
\end{enumerate}
\end{utheorem}

Finally, our last main result addresses the $q \equiv 1 \pmod 4$ case.

\begin{utheorem}\label{ThmC}
Let $q \equiv 1 \pmod 4$ and let $f:\F_q \to \F_q$. Then the following are equivalent: 
\begin{enumerate}
\item $f$ preserves positivity on $M_n(\F_q)$ for some $n \geq 3$. 
\item $f$ preserves positivity on $M_n(\F_q)$ for all $n \geq 3$. 
\item $f(0) = 0$ and $\eta(f(a)-f(b)) = \eta(a-b)$ for all $a, b \in \F_q$. 
\item $f$ is a positive multiple of a field automorphism of $\F_q$, i.e., there exist $c \in \F_q^+$ and $0 \leq \ell \leq k-1$ such that $f(x) = cx^{p^\ell}$ for all $x \in \F_q$.
\end{enumerate}
Moreover, when $q=r^2$ for some odd integer $r$, the above are equivalent to 
\begin{enumerate}
\item[(1')]$f$ preserves positivity on $M_n(\F_q)$ for some $n \geq 2$.
\end{enumerate}
\end{utheorem}

Recall that each finite field $\F_q$ with $q$ odd has an associated Paley graph $P(q)$ whose vertices are the elements of $\F_q$ and where two vertices $a,b \in \F_q$ have an edge $(a,b)$ if and only if $\eta(a-b) = 1$. The graph is directed when $q \equiv 3 \pmod 4$ and is sometimes called the Paley tournament or the Paley digraph, and is undirected when $q \equiv 1 \pmod 4$. Condition (3) in Theorems \ref{ThmB} and \ref{ThmC} can thus be rephrased as 
\begin{enumerate}
\item[(3')] $f(0) = 0$ and $f$ is an automorphism of the Paley graph $P(q)$.  
\end{enumerate}
Paley graphs play an important role in many of our proofs in the $q \equiv 1 \pmod 4$ case. Their elementary properties are reviewed in Section \ref{SSPaley}. 

Note that as the dimension $n$ of the matrices increases, the number of constraints that a positivity preserver on $M_n(\F_q)$ must satisfy quickly grows. The extreme $n=2$ case is significantly harder to resolve as there is very little structure to exploit to unveil the possible preservers. Paley graphs are particularly useful to resolve that case when $q \equiv 1 \pmod 4$ and $q=r^2$, where our arguments leverage the additional known structure of large cliques in $P(q)$ as well as ideas from finite geometry. On the other hand, when $q\equiv 1 \pmod 4$ and $q$ is a non-square, little is known about the structure of cliques in $P(q)$; in fact, even estimating the clique number of $P(q)$ itself is known to be notoriously difficult \cite{HP, Y22}. This indicates that characterizing positivity preservers on $M_2(\F_q)$ with $q \equiv 1 \pmod 4$ being a non-square is potentially very challenging.

The following corollary follows immediately from our main results, Theorems \ref{ThmA}, \ref{ThmB}, and \ref{ThmC}.

\begin{corollary}\label{Cn=3}
For any finite field $\F_q$ and any fixed $n \geq 3$, the positivity preservers on $M_n(\F_q)$ are precisely the positive multiples of the field automorphisms of $\F_q$. 
\end{corollary}

A surprising consequence of Corollary \ref{Cn=3} is the fact that if $f$ preserves positivity on $M_n(\F_q)$ for some $n \geq 3$, then $\eta(\det f[M]) = \eta(\det M)$ for any square submatrix $M$ of any matrix $A \in M_n(\F_q)$ (i.e., $f$ must preserve the ``sign'' of minors). This follows from Proposition \ref{PFrobenius} below. The analogous result does not hold for matrices in $M_n(\R)$, where positivity preservers do not generally preserve the inertia of matrices and, in particular, do not always preserve the sign of minors (see \cite{belton2023negativity} for more details).


Inspired by the above discussions, it is natural to study functions $f:\F_q \to \F_q$ that preserve the ``sign" of matrices on $M_n(\F_q)$. More precisely, we say $f:\F_q \to \F_q$ is a \emph{sign preserver} on $M_n(\F_q)$ provided that for all symmetric $A\in M_n(\F_q)$, $A$ is positive definite if and only if $f[A]$ is positive definite. Thus, a sign preserver maps positive definite matrices into themselves, and non-positive definite matrices into themselves. When $n \geq 3$, Corollary~\ref{Cn=3} implies that the sign preservers on $M_n(\F_q)$ are precisely the positive multiples of the field automorphisms of $\F_q$. When $n=2$, we prove the following theorem.

\begin{utheorem}\label{ThmD}
Let $q$ be a prime power. The sign preservers on $M_2(\F_q)$ are precisely: 
\begin{enumerate}
\item the bijective monomials, when $q$ is even.
\item the positive multiples of the field automorphisms of $\F_q$, when $q$ is odd.
\end{enumerate}
\end{utheorem}

The rest of the paper is dedicated to proving Theorems \ref{ThmA}, \ref{ThmB}, \ref{ThmC}, and \ref{ThmD}. Section \ref{Sprelim} contains preliminary results including statements of classical results from finite field theory that are needed in the proofs, a discussion of the properties of positive definite matrices with entries in a finite field, and preliminary results on entrywise preservers over finite fields. Sections \ref{Seven}, \ref{S3mod4}, and \ref{S1mod4} address the even case (Theorem \ref{ThmA}), the $q \equiv 3\pmod 4$ case (Theorem \ref{ThmB}), and the $q \equiv 1 \pmod 4$ case (Theorem \ref{ThmC}), respectively. Section \ref{S1mod4} also contains the proof of Theorem \ref{ThmD}. Section \ref{S1mod4square} addresses the $q = r^2$ case (Part (1') in Theorem \ref{ThmC}). Section \ref{Smisc} contains an alternative approach to prove some of our results. Concluding remarks are given in Section \ref{Sconc}. 

\section{Preliminaries}\label{Sprelim}

\subsection{Finite fields}

We first recall the characterization of automorphisms of finite fields. 

\begin{theorem}[{\cite[Theorem 2.21]{lidl1997finite}}]\label{Tauto}
Let $q = p^k$. Then the distinct automorphisms of $\F_q$ are exactly the mappings $\sigma_0, \sigma_1, \dots, \sigma_{k-1}$ defined by $\sigma_\ell(x) = x^{p^\ell}$.  
\end{theorem}
\noindent In particular, $(x+y)^{p^\ell} = \sigma_\ell(x+y) = \sigma_\ell(x) + \sigma_\ell(y) = x^{p^\ell} + y^{p^\ell}$ in a field of characteristic $p$.

Next, recall some elementary facts about permutation polynomials over $\F_q$, i.e., polynomials that are bijective on $\F_q$.  

\begin{theorem}[{\cite[Theorem 7.8]{lidl1997finite}}]\label{Tperm}\hfill
\begin{enumerate}
    \item Every non-constant linear polynomial over $\F_q$ is a permutation polynomial of $\F_q$. 
    \item The monomial $x^n$ is a permutation polynomial of $\F_q$ if and only if $\gcd(n, q-1) = 1$.
\end{enumerate}
\end{theorem}

The following simple facts will be useful later. We provide a short proof for completeness. 
\begin{proposition}\label{Pdecomp}
Let $\F_q$ be a finite field of odd characteristic. Then the following are equivalent:
\begin{enumerate}
\item $q \equiv 3 \pmod 4$.
\item $-1$ is not a square in $\F_q$.
\item $\F_q^- = -\F_q^+$. 
\item Every element in $\F_q^+$ has a {\it unique} positive square root.
\end{enumerate}
\end{proposition}
\begin{proof}
The equivalence between (1) and (2) is folklore (see e.g.~\cite[Corollary II.2.2]{koblitz1994course}). The equivalence between (2) and (3) follows immediately from $\eta(-x) = \eta(-1)\eta(x)$.

Now, suppose (3) holds. Let $x \in \F_q^+$, say $x = y^2$. Then $y$ and $-y$ are exactly the square roots of $x$ because every element in $\F_q$ has at most $2$ square roots. Since only one of these is positive, the positive square root of $x$ must be unique. Finally, suppose (4) holds. Since $1^2 = (-1)^2 = 1$, both $1$ and $-1$ are square roots of $1$ in $\F_q$. Since $1 \in \F_q^+$ the uniqueness implies that $-1 \in \F_q^-$ and (3) follows.  
\end{proof}
\noindent When $q$ is even, since $x \mapsto x^2$ is a bijective map,  every non-zero element also has a unique positive square root. When $q$ is even or $q \equiv 3 \pmod{4}$, we denote the unique positive square root of $x \in \F_q^+$ by $\sqrt{x}$ or by $x^{1/2}$. We also define $\sqrt{0} = 0$.

The next classical lemma shows that two polynomials in $\F_q[x]$ coincide as functions, i.e., when evaluated at every point of $\F_q$, if and only if they are equal as polynomials modulo $x^q-x$. 

\begin{lemma}[{\cite[Lemma 7.2]{lidl1997finite}}]\label{Lidl_Lemma}
    For $g(x),h(x) \in \F_q[x]$ we have $g(c) = h(c)$ for all $c\in \F_q$ if and only if $g(x) \equiv h(x) \pmod{x^q-x}$. 
\end{lemma}

Notice that every function $f: \F_q \to \F_q$ can be written as an interpolation polynomial of degree at most $q-1$. When studying entrywise positivity preservers, we can thus assume, without loss of generality, that $f$ is a polynomial of degree at most $q-1$.  

We also recall the following well-known theorem, due to  Carlitz~\cite{carlitz1960theorem}.

\begin{theorem}[\cite{carlitz1960theorem}]\label{Tcarlitz}
Let $q=p^k$, where $p$ is an odd prime. Let $f:\F_q \to \F_q$ such that $f(0) = 0$, $f(1) = 1$, and $\eta(f(a)-f(b)) = \eta(a-b)$ for all $a, b \in \F_q$. Then there is $0 \leq \ell \leq k-1$, such that $f(x) = x^{p^\ell}$ for all $x \in \F_q$. 
\end{theorem}

\subsection{Positive definite matrices over finite fields}\label{sec:pdm}
For real symmetric or complex Hermitian matrices, it is well-known that many natural notions of positive definiteness coincide. Any of the following equivalent conditions can be used to define positive definiteness.  

\begin{proposition}[{\cite[Chapter 7]{horn2012matrix}}]\label{Ppd}
Let $A \in M_n(\C)$ be a Hermitian matrix. Then the following are equivalent: 
\begin{enumerate}
\item $z^* A z > 0$ for all non-zero $z \in \C^n$.
\item All eigenvalues of $A$ are positive. 
\item The sesquilinar form $Q(z,w) = z^* Aw$ forms an inner product.
\item $A$ is the Gram matrix of linearly independent vectors. 
\item All leading principal minors of $A$ are positive. 
\item $A$ has a unique Cholesky decomposition. 
\end{enumerate}
\end{proposition}

As shown by Cooper, Hanna, and Whitlatch~\cite{cooper2022positive}, the situation is very different for matrices over finite fields. For example, the standard definition of positive definiteness via quadratic forms (as in Proposition \ref{Ppd}(1)) does not yield a useful notion over finite fields.

\begin{proposition}[{\cite[Proposition 4]{cooper2022positive}}]\label{PquadFormZero}
Let $\F_q$ be a finite field, let $n \geq 3$, and let $A \in M_n(\F_q)$. Define $Q: \F_q^n \to \F_q$ by $Q(x) = x^TAx$. Then there exists a non-zero vector $v \in \F_q^n$ so that $Q(v) = 0$.
\end{proposition}

\noindent In fact, more can be said about the range of the quadratic form associated to a positive definite matrix. 

\begin{proposition}\label{PquadRange}
Let $n \geq 2$ and let $A \in M_n(\F_q)$ be a positive definite matrix. Then the range of the quadratic form $Q(x) = x^T A x$ is $\F_q$, i.e., $\{x^T A x : x \in \F_q^n\} = \F_q.$
\end{proposition}
\begin{proof}
Suppose first $n=2$. Let 
\[
A = \begin{pmatrix}
    a & b \\
    b & c
\end{pmatrix} \in M_2(\F_q)
\]
be positive definite. Then $a \in \F_q^+$ and $ac-b^2 \in \F_q^+$. In particular, $c-b^2a^{-1} \in \F_q^+$. For $x = (x_1, x_2)^T \in \F_q^2$, consider the quadratic form 
\[
Q(x) = x^T A x = ax_1^2 + 2bx_1 x_2 + c x_2^2. 
\]
Completing the square, we obtain
\[
Q(x) = a(x_1 + ba^{-1}x_2)^2 + (c-b^2a^{-1}) x_2^2. 
\]
Setting $y_1 := a^{1/2}\left(x_1 + ba^{-1}x_2\right)$ and $y_2 := (c-b^2a^{-1})^{1/2} x_2$ yields the equivalent diagonal quadratic form 
\[
\widetilde{Q}(y) = y_1^2 + y_2^2
\]
having the same range as $Q$. Let $S$ be the set of squares in $\F_q$; then $|S|\geq \frac{q+1}{2}$. Thus, for each $x\in \F_q$, $S \cap (x-S)\neq \emptyset$, that is, $x$ can be written as the sum of two squares. It follows that the range of $Q$ is $\F_q$. 

Suppose now $n \geq 3$. Let $\widetilde{A} \in M_2(\F_q)$ be the $2 \times 2$ leading principal submatrix of $A$. Then $\widetilde{A}$ is positive definite. Letting $x := (\widetilde{x}^T, {\bf 0}_{1 \times (n-2)})^T \in \F_q^n$ with $\widetilde{x} \in \F_q^2$, we obtain $x^T A x = \widetilde{x}^T \widetilde{A}\widetilde{x}.$ The result now follows from the $n=2$ case.
\end{proof}

When $q$ is even or $q \equiv 3 \pmod 4$, some of the classical real/complex positivity theory can be recovered. Recall that a symmetric matrix $A \in M_n(\F_q)$ is said to have a {\it Cholesky decomposition} if $A = LL^T$ for some lower triangular matrix $L \in M_n(\F_q)$ with positive elements on its diagonal. When $q$ is even or $q \equiv 3 \pmod 4$, it is known that the positivity of the leading principal minors of a matrix in $M_n(\F_q)$ is equivalent to the existence of a Cholesky decomposition.
\begin{theorem}[{\cite[Theorem 16, Corollary 24]{cooper2022positive}}]\label{Tcharac}
Let $A \in M_n(\F_q)$ be a symmetric matrix.
\begin{enumerate}
    \item If $A$ admits a Cholesky decomposition, then all its leading principal minors are positive.
    \item If $q$ is even or $q \equiv 3 \pmod 4$ and all the leading principal minors of $A$ are positive, then $A$ admits a Cholesky decomposition. 
\end{enumerate} 
\end{theorem}

\noindent We note however that the equivalence fails in general when $q \equiv 1 \pmod 4$.  
\begin{proposition}
Let $q \equiv 1 \pmod 4$. Then there exists a positive definite matrix $A \in M_2(\F_q)$ that does not admit a Cholesky decomposition. 
\end{proposition}
\begin{proof}
For $x \in \F_q^*$, let 
\[
A(x) := \begin{pmatrix}
1 & x \\
x & 0
\end{pmatrix}. 
\]
Then $A(x)$ is positive definite since $-1 \in \F_q^+$. Suppose $A(x) = LL^T$, say 
\[
A(x) = \begin{pmatrix}
1 & x \\
x & 0
\end{pmatrix} = \begin{pmatrix}
a & 0 \\
b & c
\end{pmatrix} \begin{pmatrix}
a & b \\
0 & c
\end{pmatrix} = \begin{pmatrix}
a^2 & ab \\
ab & b^2 + c^2
\end{pmatrix}
\]
with $a, c \in \F_q^+$. Then $a = \pm 1$, $b = \pm x$ and $c^2 = -b^2 = -x^2$. Thus $c \in \{i x, -ix\}$ where $i$ denotes a square root of $-1$ in $\F_q$. We can then pick $x \in \F_q^*$ such that $\eta(c) = \eta(i)\eta(x) = -1$. Such a choice of $x$ forces $c \not\in \F_q^+$ and therefore the Cholesky decomposition of $A(x)$ does not exist.
\end{proof}

\begin{remark}
We note that, when $q$ is even or $q \equiv 3 \pmod 4$, the authors of \cite{cooper2022positive} define a symmetric matrix in $M_n(\F_q)$ to be positive definite if it admits a Cholesky decomposition. As Theorem \ref{Tcharac} shows, this definition coincides with ours. We note, however, that verifying if a matrix admits a Cholesky decomposition is not as straightforward as computing its leading principal minors. This is our motivation for adopting Definition \ref{Dpd}. 
\end{remark}

As discussed in the proof of Proposition~\ref{PquadRange}, every element in a finite field can be written as a sum of two squares. As a consequence, sums of positive definite matrices are not always positive definite. Similarly, a Gram matrix $A = MM^T$ with $M \in M_{n \times m}(\F_q)$ is not always positive definite (take, for example, $M = (x,y) \in M_{1 \times 2}(\F_q)$ with $x^2 + y^2 \not\in \F_q^+$). Many other standard properties of positive definite matrices over $\R$ or $\C$ fail for finite fields. For example, a positive definite matrix may not have positive eigenvalues and the Hadamard product of two positive definite matrices is not always positive definite. See \cite[Section 3]{cooper2022positive} for more details. As mentioned above, the behavior of the quadratic form of a positive definite matrix is also different over finite fields (see Proposition \ref{PquadRange}). The reader who is accustomed to working with positive definite matrices over the real or the complex field must thus take great care when moving to the finite field world. 

\subsection{Entrywise preservers}
We now turn our attention to entrywise positivity preservers on $M_n(\F_q)$. Recall that every function $f: \F_q \to \F_q$ coincides with a polynomial of degree at most $q-1$ (Lemma \ref{Lidl_Lemma}). Unless otherwise specified, we therefore assume below that $f$ is such a polynomial. 

When $n=1$, the positivity preservers are precisely the functions $f: \F_q \to \F_q$ such that $f(\F_q^+) \subseteq \F_q^+$. In characteristic $2$, we have $\F_q^+ = \F_q^*$ and the positivity condition reduces to $0 \not\in f(\F_q^*)$. There are $(q-1)^{q-1} \times q$ such maps. In odd characteristic, the number of positivity preservers is $\left(\frac{q-1}{2}\right)^{\frac{q-1}{2}} \times q^{\frac{q+1}{2}}$. Any such map can be explicitly written using an interpolation polynomial. We therefore focus on the $n \geq 2$ case below.

We next obtain a family of maps that preserves positivity for matrices with entries in any finite field. 

\begin{proposition}\label{PFrobenius}
Let $q = p^k$ and let $f(x) = x^{p^l}$ be an automorphism of $\F_q$. Then for any $n \geq 1$ and any $A \in M_n(\F_q)$, we have $\det f[A] = f(\det A)$. In particular, all the positive multiples of the field automorphisms of $\F_q$ preserve positivity on $M_n(\F_q)$ for all $n \geq 1$. 
\end{proposition}
\begin{proof}
    Let $A = (a_{ij}) \in M_n(\F_q)$. By the Leibniz formula for the determinant and Theorem \ref{Tauto}, 
    \[
    \det f[A] = \sum_{\sigma \in S_n} \sgn(\sigma) a_{1,\sigma(1)}^{p^\ell} a_{2,\sigma(2)}^{p^\ell} \dots a_{n, \sigma(n)}^{p^\ell} = \left(\sum_{\sigma \in S_n} \sgn(\sigma) a_{1, \sigma(1)} a_{2, \sigma(2)}\dots a_{n, \sigma(n)}\right)^{p^\ell} = f(\det A).
\]
    In particular, suppose $A$ is positive definite and let $A_r$ denote the leading $r \times r$ principal submatrix of $A$. By Definition \ref{Dpd}, $\det A_r = \mu^2$ for some $\mu \in \F_q^*$ and so $\det f[A_r] = f(\mu^2) = \left(\mu^2\right)^{p^l} = \left(\mu^{p^l}\right)^2 \in \F_q^+$. Since the above holds for any $1 \leq r \leq n$, the matrix $f[A]$ is positive definite. Clearly, multiplying $f$ by $c\in \F_q^+$ also yields a positivity preserver.
\end{proof}

Next, we provide some simple necessary conditions for preserving positivity on $M_n(\F_q)$.
\begin{lemma}\label{lem:+}
Let $n\geq 2$ be an integer, $q$ a prime power, and let $f:\F_q \to \F_q$ be a positivity preserver over $M_n(\F_q)$. Then $f(\F_q^+) \subseteq \F_q^+$.
\end{lemma}
\begin{proof}
Let $a \in \F_q^+$. Since $aI_n$ is positive definite, so is $f[aI_n]$. In particular, $f(a) \in \F_q^+$.
\end{proof}

\begin{lemma}\label{L:Charq}
    Let $q$ be a prime power with $q$ even or $q \equiv 3 \pmod 4$ and let $f: \F_q \to \F_q$ be a positivity preserver on $M_2(\F_q)$. Then:
    \begin{enumerate}
    \item The restriction of $f$ to $\F_q^+$ is a bijection of $\F_q^+$ onto itself.
    \item $f(0) = 0$.
    \end{enumerate}
\end{lemma}
\begin{proof}
When $q=3$, the result follows immediately by applying $f$ to $I_2$. Now assume $q > 3$. Let $a, b \in \F_q^+$ with $a \ne b$. Thus, either $a-b \in \F_q^+$ or $b-a \in \F_q^+$. Say $a-b \in \F_q^+$ without loss of generality. Thus, the matrix
\[
A = \begin{pmatrix}
    b & b \\
    b & a
\end{pmatrix}
\]
is positive definite. Note that $f(a), f(b) \in \F_q^+$ by Lemma~\ref{lem:+}. By assumption, $f[A]$ is also positive definite. Hence, $\det f[A] = f(b)(f(a)-f(b)) \in \F_q^+$. In particular, $f(a) \ne f(b)$. This proves that $f$ is an injective map on $\F_q^+$, and is therefore a bijection from $\F_q^+$ onto itself. This proves (1). 

Now, suppose $f(0) = c$ where $c \in \F_q^+$. By the first part, there exists $a \in \F_q^+$ such that $f(a) = c$. Since the matrix $aI_2$ is positive definite so is $f[aI_2]$. However, 
\[
f[a I_2] = \begin{pmatrix}
c & c \\
c & c
\end{pmatrix} 
\]
is not positive definite. If instead $f(0) \in \F_q^-$, then $c := -f(0) \in \F_q^+$. Now repeat the above argument to get $\det f[a I_2] = 0$, again a contradiction. Thus,  $f(0) = 0$. 
\end{proof}

\noindent The proof of Lemma \ref{L:Charq} does not work when $q \equiv 1 \pmod 4$. However, the following lemma shows that $f$ needs to be injective on certain subsets of $\F_q^+$.
\begin{lemma}\label{lem:sq}
Let $q$ be a prime power with $q \equiv 1 \pmod 4$. Let $f:\F_q \to \F_q$ be a positivity preserver over $M_2(\F_q)$. Let $a,b \in \F_q$ such that $a-b \in \F_q^+$. If $a \in \F_q^+$ or $b \in \F_q^+$, then $f(a)-f(b) \in \F_q^+$.    
\end{lemma}

\begin{proof}
Without loss of generality, assume that $b \in \F_q^+$. Consider the matrix 
$$
A=
\begin{pmatrix} 
b&b\\
b&a
\end{pmatrix}
$$
It has determinant $b(a-b)$ and thus it is positive definite. Under the map $f$, we have $f(b)(f(a)-f(b)) \in \F_q^+$. By Lemma~\ref{lem:+}, we have $f(b)\in \F_q^+$ and thus $f(a)-f(b)\in \F_q^+$.
\end{proof}

\begin{lemma}\label{lem:nonzero}
Let $q$ be a prime power with $q \equiv 1 \pmod 4$ and let $f:\F_q \to \F_q$ be a positivity preserver over $M_2(\F_q)$. If $f(0)=0$, then $f(x)\neq 0$ for each $x \in \F_q^*$.    
\end{lemma}
\begin{proof}
Assume otherwise that there is $x\in \F_q^*$ such that $f(x)=0$. Consider the matrix
\[
A = \begin{pmatrix}
1 & x \\
x & 0
\end{pmatrix}.
\]
Clearly, $A$ is positive definite since $-1 \in \F_q^+$. However, $f[A]$ is singular, a contradiction.   
\end{proof}     

\subsection{Distribution of elements in translations of $\F_q^+$}
We now prove several lemmas on the distribution of elements in translations of $\F_q^+$ using standard character sum estimates. These lemmas will be useful in the proof of our main results. Recall that $\eta$ denotes the quadratic character of $\F_q$ (see Equation~\eqref{EquadChar}). 

\begin{lemma}\label{Ldoublesquare}
Let $\F_q$ be a finite field with $q \equiv 3 \pmod 4$. Fix $a \in \F_q^*$, and define 
$a+\F_q^+ := \{a + y : y \in \F_q^+\}.$ Then $|\F_q^+ \cap (a+\F_q^+)| = \frac{q-3}{4}$. 
\end{lemma}
\begin{proof}
For $a \in \F_q^*$, we have
\begin{align*}
|\F_q^+ \cap (a+\F_q^+)| &= \sum_{x \in \F_q \setminus \{0, -a\}} \frac{\eta(x)+1}{2}\cdot \frac{\eta(x+a)+1}{2} \\
&= \frac{1}{4}\left( \sum_{x \in \F_q} \eta(x)\eta(x+a) + \sum_{x \in \F_q \setminus \{-a\}} \eta(x) +  \sum_{x \in \F_q \setminus \{0\}} \eta(x+a) + \sum_{x \in \F_q \setminus \{0, -a\}} 1 \right)\\
&= \frac{1}{4}\left(-1 - \eta(-a) - \eta(a) + q-2 \right)= \frac{q-3}{4}, 
\end{align*}
where for the first term, we use \cite[Theorem 5.48]{lidl1997finite}. 
\end{proof}

Given three distinct elements $a,b,c$ in $\F_q$, let $t_q(a,b,c)$ be the number of $x \in \F_q$ such that $\eta(x-a)=\eta(x-b)=\eta(x-c)=1$. The following lemma provides estimates on $t_q(a,b,c)$ using a standard application of Weil's bound. We note that $t_q(a,b,c)$ can also be estimated directly using \cite[Exercise 5.64]{lidl1997finite}. However, for our purposes, we need a more careful analysis that handles the case where $q$ is relatively small. A similar computation also appeared in \cite{BM22} when $q \equiv 1 \pmod 4$.

\begin{lemma}\label{lem:triple}
Let $q$ be an odd prime power and let $t_q=t_q(a,b,c)$ be as above. Then
\begin{enumerate}
    \item $t_3,t_5 \in \{0\}$, $t_7,t_9,t_{11}\in \{0,1\}$, $t_{13},t_{17}\in \{0,1,2\}$, $t_{19},t_{23} \in \{1,2,3\}$, and $t_{25} \in \{0,2,3,4\}$.
    \item If $q \geq 27$, then $0<t_q<\frac{q-5}{4}$.
\end{enumerate}
\end{lemma}
\begin{proof}
Observe that
\[
S := t_q(a,b,c) = \sum_{x \in \F_q \setminus \{a,b,c\}} \frac{\eta(x-a)+1}{2}\cdot \frac{\eta(x-b)+1}{2}\cdot \frac{\eta(x-c)+1}{2}.
\]
Thus, 
\begin{align*}
8S = \sum_{x \in \F_q \setminus \{a,b,c\}} & \bigg(\eta(x-a)\eta(x-b)\eta(x-c) + \eta(x-a)\eta(x-b) \\
&+\eta(x-a)\eta(x-c)+\eta(x-b)\eta(x-c)+\eta(x-a)+\eta(x-b)+\eta(x-c)+1\bigg). 
\end{align*}
We examine each term separately. First, using Weil's bound (see for example \cite[Theorem 5.41]{lidl1997finite}), 
\[
\bigg|\sum_{x \in \F_q \setminus \{a,b,c\}} \eta(x-a)\eta(x-b)\eta(x-c)\bigg| = \bigg|\sum_{x \in \F_q} \eta(x-a)\eta(x-b)\eta(x-c) \bigg|\leq 2 \sqrt{q}. 
\]

Next, by \cite[Theorem 5.48]{lidl1997finite}, we have
$$
\sum_{x \in \F_q \setminus \{a,b,c\}} \eta(x-a)\eta(x-b)=-\eta(c-a)\eta(c-b)+\sum_{x \in \F_q} \eta(x-a)\eta(x-b)=-\eta(c-a)\eta(c-b)-1.
$$
Similarly, we have 
$$
\sum_{x \in \F_q \setminus \{a,b,c\}} \eta(x-b)\eta(x-c)=-\eta(a-b)\eta(a-c)-1, \quad \sum_{x \in \F_q \setminus \{a,b,c\}} \eta(x-c)\eta(x-a)=-\eta(b-c)\eta(b-a)-1.
$$
Finally, we have
$$
\sum_{x \in \F_q \setminus \{a,b,c\}} \eta(x-a)=-\eta(b-a)-\eta(c-a)+\sum_{x \in \F_q} \eta(x-a)=-\eta(b-a)-\eta(c-a), 
$$
and similarly, 
$$
\sum_{x \in \F_q \setminus \{a,b,c\}} \eta(x-b)=-\eta(a-b)-\eta(c-b), \sum_{x \in \F_q \setminus \{a,b,c\}} \eta(x-c)=-\eta(a-c)-\eta(b-c).
$$
Combining all the above estimates, we obtain
\begin{align*}
q-2\sqrt{q}-15\leq 8S \leq q+3+2\sqrt{q}.
\end{align*}
This proves part (2) along with bounds for $t_q$ when $q\geq 27$. The refinements in (1) are readily verified by computer.
\end{proof}

\section{Even characteristic}\label{Seven}
In this section, we always assume $q = 2^k$ for some integer $k \geq 1$. Recall that in this case, $\F_q^+ = \F_q^*$. Positive definiteness thus reduces to the non-vanishing of the leading principal minors. We break down the proof of Theorem \ref{ThmA} into two parts: the $n=2$ case (Theorem \ref{Tchar2}) and the $n \geq 3$ case (Theorem \ref{Tchar2M3}). 

\begin{theorem}\label{Tchar2}
Let $q = 2^k$ for some $k \geq 1$ and let $f: \F_q \to \F_q$. Then the following are equivalent:
\begin{enumerate}
\item $f$ preserves positivity on $M_2(\F_q)$.
\item $f(0) = 0$, $f$ is bijective, and $f(\sqrt{xy})^2 = f(x)f(y)$ for all $x, y \in \F_q$.
\item There exist $c \in \F_q^*$ and $1 \leq n \leq q-1$ with $\gcd(n, q-1) = 1$ such that $f(x) = cx^n$ for all $x \in \F_q$.
\end{enumerate}
\end{theorem}
\begin{proof}
\noindent $(1) \implies (2)$. Suppose (1) holds. Then $f(0) = 0$ and $f$ is bijective on $\F_q^+ = \F_q^*$ by Lemma \ref{L:Charq}. Thus, $f$ is bijective on $\F_q$. Fix $x, y \in \F_q^*$ and consider the matrix
\[
A(z) = \begin{pmatrix}
    x & \sqrt{xy}z \\
    \sqrt{xy}z & y
\end{pmatrix} \qquad (z \in \F_q).
\]
Observe that $A(z)$ is positive definite if and only if $z \ne 1$. Thus, for any $z \ne 1$, $f[A(z)]$ is positive definite and so 
\[
\det f[A(z)] = f(x)f(y)-f(\sqrt{xy}z)^2 \ne 0.
\]
Hence, for all $z \ne 1$, 
\begin{equation}\label{EnotEqual}
f(\sqrt{xy}z)^2 \ne f(x)f(y).
\end{equation}
Since $f$ and the $x \mapsto x^2$ map are bijections, there exists a unique $w \in \F_q$ such that $f(w)^2 = f(x)f(y)$. Also, the map $z \mapsto \sqrt{xy} z$ is a bijection of $\F_q$. Using equation \eqref{EnotEqual}, we conclude that $w = \sqrt{xy}$ and so $f(\sqrt{xy})^2 = f(x)f(y)$. The expression $f(\sqrt{xy})^2 = f(x)f(y)$ also holds trivially when $x = 0$ or $y = 0$ since $f(0) = 0$. This proves (2). 
\medskip

\noindent $(2) \implies (3)$. Suppose (2) holds and let $f(x) = \sum_{k=1}^{q-1} a_k x^k$ without loss of generality. Note that 
\[
f(\sqrt{xy})^2 = \left(\sum_{k=1}^{q-1} a_k (\sqrt{xy})^k\right)^2 = \sum_{k=1}^{q-1} a_k^2 x^k y^k.
\]
Next, we compute
\[
f(x)f(y) = \left(\sum_{i=1}^{q-1} a_i x^i\right)\left(\sum_{j=1}^{q-1} a_j x^j\right) = \sum_{k=1}^{q-1} a_k^2 x^k y^k + \sum_{1 \leq i < j \leq q-1} a_i a_j (x^iy^j + x^j y^i). 
\]
Since $f(\sqrt{xy})^2 = f(x)f(y)$ for all $x, y \in \F_q$, we conclude that 
\[
Q(x,y) := \sum_{1 \leq i < j \leq q-1} a_i a_j (x^iy^j + x^j y^i) = 0
\]
for all $x, y \in \F_q$. Now, for any fixed $y$,
\[
Q(x,y) = \sum_{k=1}^{q-1} \left(\sum_{\substack{1 \leq j \leq q-1 \\ j \ne k}} a_j a_k y^j\right) x^k 
\]
is a polynomial in $x$ of degree at most $q-1$ that is identically $0$ on $\F_q$. Therefore, by Lemma \ref{Lidl_Lemma},  
\[
\sum_{\substack{1 \leq j \leq q-1 \\ j \ne k}} a_j a_k y^j = 0 \qquad (1 \leq k \leq q-1).
\]
Since this is true for all $y \in \F_q$ and since the above expression is a polynomial of degree at most $q-1$, we conclude that $a_j a_k = 0$ for all $j \ne k$. This proves $f(x)$ is a monomial and so $f(x) = c x^n$ for some $1 \leq n \leq q-1$. Since $f$ is bijective, Theorem \ref{Tperm}(2) implies that $c \neq 0$ and $\gcd(n, q-1) = 1$.
\medskip

\noindent $(3) \implies (1)$. Suppose (3) holds and let 
\[
A = \begin{pmatrix}
    u & v \\
    v & w
\end{pmatrix}
\]
be an arbitrary positive definite matrix in $M_2(\F_q)$, i.e., $u \ne 0$ and $uw \ne v^2$. Clearly, $f(u) = cu^n \ne 0$. Moreover, since $x \mapsto x^n$ is injective on $\F_q$, we have $u^n w^n \ne v^{2n}$ and so 
\[
\det f[A] = c^2 u^n w^n - c^2 v^{2n} \ne 0.
\]
This proves $f$ preserves positivity on $M_2(\F_q)$ and so (1) holds. This concludes the proof.
\end{proof}

We now describe the entrywise positivity preservers on $M_3(\F_q)$. 

\begin{theorem}\label{Tchar2M3}
Let $q = 2^k$ and let $f: \F_q \to \F_q$. Then the following are equivalent:
\begin{enumerate}
\item $f$ preserves positivity on $M_3(\F_q)$.
\item There exist $c \in \F_q^*$ and $0 \leq \ell \leq k-1$ such that $f(x) = cx^{2^\ell}$ for all $x \in \F_q$.
\end{enumerate}
\end{theorem}
\begin{proof}
That $(2) \implies (1)$ follows from Proposition \ref{PFrobenius}. Now, suppose (1) holds. By embedding $2 \times 2$ positive definite matrices $A$ into $M_3(\F_q)$ via
\[
\begin{pmatrix}A & {\bf 0}_{2 \times 1} \\
{\bf 0}_{1 \times 2} & 1
\end{pmatrix} \in M_3(\F_q), 
\]
it follows by Theorem \ref{Tchar2} that $f(x) = c x^n$ for all $x \in \F_q$, where $c \in \F_q^*$ and $1 \leq n \leq q-1$ is such that $\gcd(n, q-1) = 1$. Without loss of generality, we assume that $c=1$. It suffices to show that the only exponents $n$ that preserve positivity on $M_3(\F_q)$ are powers of $2$.  

For $x, y \in \F_q$, let 
\[
A(x,y) = \begin{pmatrix}
1 & x & y \\
x & 1 & 0 \\
y & 0 & 1
\end{pmatrix}.
\]
The matrix $A(x,y)$ is positive definite if and only if $x \ne 1$ and $\det A = 1-x^2-y^2 \ne 0$. Notice that, using the fact that $-1 = 1$ in $\F_q$,  
\[
\det A(x,y) = 0 \iff x^2 + y^2 = 1 \iff (x+y)^2 = 1 \iff x+y=1.
\]
Similarly, $\det f[A] = 1-x^{2n}-y^{2n}$ and so 
\[
\det f[A(x,y)] = 0 \iff x^{2n}+y^{2n} =1 \iff (x^n+y^n)^2 = 1 \iff x^n + y^n = 1.
\]
Suppose $n$ is not a power of $2$. We will prove that there exist $x_0, y_0 \in \F_q$ such that $A(x_0,y_0)$ is positive definite, but $f[A(x_0,y_0)]$ is not positive definite. In order to do so, it suffices to prove the existence of $x_0, y_0 \in \F_q$ such that $x_0 \ne 1$, $x_0+y_0 \ne 1$, and $x_0^n + y_0^n = 1$. Indeed, consider the two sets: 
\[
S_1 = \{(x,y) \in \F_q^2 : x + y = 1\}, \qquad S_2 = \{(x,y) \in \F_q^2 : x^n + y^n = 1\}.
\]
Clearly, $|S_1| = q$ since for every $x \in \F_q$, there is a unique $y \in \F_q$ such that $x +y = 1$. Recall that the map $x \mapsto x^n$ is a bijection since $\gcd(n, q-1) 
= 1$ (Theorem \ref{Tperm}(2)). It follows that $|S_2| = q$ as well. Now, suppose the desired pair $x_0, y_0$ does not exist. Then for every $(x,y) \in S_2$, either $x = 1$ or $x+y = 1$. But if $x = 1$ then $y=0$ (since $(x,y) \in S_2$) and so $x+y = 1$. In all cases, $(x,y) \in S_1$ and it follows that $S_2 \subseteq S_1$. Since the two sets have the same cardinality, we conclude that $S_1 = S_2$. Thus, 
\[
x^n + y^n = 1 \iff x + y = 1.
\]
Now it is easy to verify that this implies the map $f(x) = x^n$ is an automorphism of $\F_q$. By Theorem \ref{Tauto}, we therefore must have $n \equiv 2^\ell \pmod{q-1}$ for some $\ell$. This is impossible since $1 \leq n \leq q-1$ and $n$ is not a power of $2$. We therefore conclude that there exist $x_0, y_0 \in \F_q$ such that $x_0 \ne 1$, $x_0+y_0 \ne 1$, and $x_0^n + y_0^n = 1$. This proves $(1) \implies (2)$.  
\end{proof}

Using Theorem \ref{Tchar2} and \ref{Tchar2M3}, we immediately obtain Theorem \ref{ThmA}. 

\begin{proof}[Proof of Theorem \ref{ThmA}]
The $n=2$ case is Theorem \ref{Tchar2}. Consider now the $n\geq 3$ case. Clearly (b) $\implies$ (a). Suppose (a) holds. If $n > 3$, then using matrices of the form $A \oplus I_{n-3}$ with $A \in M_3(\F_q)$, we conclude that $f$ preserves positivity on $M_3(\F_q)$. Theorem \ref{Tchar2M3} then implies that (c) holds. The (c) $\implies$ (b) implication is Proposition \ref{PFrobenius}. 
\end{proof}

\section{Odd characteristic: $q \equiv 3 \pmod 4$}\label{S3mod4}
We now move to the case where $q \equiv 3 \pmod 4$. We break down the proof of Theorem \ref{ThmB} into several lemmas. The $n=2$ case of the theorem is considerably more difficult to prove as very little structure is available to work with. Most of the results below rely on indirect algebraic/combinatorial arguments to obtain relevant properties of the preservers. When $n \geq 3$, although the result follows from the $n=2$ case, the supplementary structure of $3 \times 3$ matrices can be used to give a shorter proof of the theorem. We first show how to obtain the $n=2$ case, and then explain how a simpler approach can be used to deduce the $n \geq 3$ case.

\begin{lemma}\label{lbijective}
Let $\F_q$ be a finite field with $q \equiv 3 \pmod 4$ and let $f: \F_q \to \F_q$ preserve positivity on $M_2(\F_q).$ Then $f(0) = 0$ and $f$ is bijective on $\F_q^+$ and on $\F_q^-$ (and hence on $\F_q$). 
\end{lemma}
\begin{proof}
By Lemma \ref{L:Charq}, the function $f$ satisfies $f(0) = 0$ and its restriction to $\F_q^+$ is a bijection onto $\F_q^+$. We will conclude the proof by proving that $f(\F_q^-) \subset \F_q^-$ and that $f$ is injective on $\F_q^-$. When $q=3$, this follows immediately by applying $f$ to the positive definite matrix $\begin{pmatrix}
1 & 1 \\
1 & -1
\end{pmatrix}$. We therefore assume below that $q > 3$.
\medskip

\noindent {\bf Step 1: $f(\F_q^-) \subset \F_q^-$}.  Suppose for a contradiction that $f(-b) \in \F_q^+$ for some $b \in \F_q^+$. Since $f$ is bijective from $\F_q^+$ onto itself, $f(-b) = f(a)$ for some $a \in \F_q^+$. Let $y := f(a) = f(-b)$. For $x \in \F_q^+$, consider the matrix 
\[
A(x) = \begin{pmatrix}
    x & a \\
    a & -b
\end{pmatrix}.
\]
Observe that $\det f[A(x)] = f(x)f(-b)-f(a)^2 = y\left(f(x)-y\right)$. Since $y = f(a) \in \F_q^+$, it follows that 
\[
f[A(x)] \textrm{ is positive definite } \iff f(x)-y \in \F_q^+. 
\]
Define
\[
L := \{x \in \F_q^+ : f(x)-y \in \F_q^+\}.
\]
Since $f$ is bijective on $\F_q^+$, by Lemma \ref{Ldoublesquare}, we have $|L| = \frac{q-3}{4}$. Now, let 
\[
M := \{x \in \F_q^+ : -bx - a^2 \in \F_q^+\}. 
\]
Observe that 
\[
A(x) \textrm{ is positive definite } \iff x \in M.
\]
We claim $|M| = \frac{q+1}{4} > \frac{q-3}{4}$. Indeed, 
\begin{align*}
    x \in M &\iff x \in \F_q^+ \textrm{ and } -bx-a^2 \in \F_q^+ 
    \iff x \in \F_q^+ \textrm{ and } x + a^2b^{-1} \in \F_q^-.
\end{align*}
Using Lemma \ref{Ldoublesquare} again, the cardinality of the set
\[
S := \{x \in \F_q^+ : x + a^2b^{-1} \in \F_q^+\}
\]
is $|S| = \frac{q-3}{4}$. Observe that $x+a^2b^{-1} = 0$ implies $x = -a^2b^{-1} \in \F_q^-$. It follows that $M = \F_q^+ \setminus S$ and so 
\[
|M| = \frac{q-1}{2} - \frac{q-3}{4} = \frac{q+1}{4}.
\]
Therefore, there exists $x^* \in M$ such that $x^* \not\in L$. Thus, $A(x^*)$ is positive definite, but $f[A(x^*)]$ is not positive definite, contradicting the assumption of the theorem. We therefore conclude that $f(\F_q^-) \subseteq \F_q^- \cup \{0\}$. Finally, suppose $f(-b) = 0$ for some $b \in \F_q^+$. Taking any $x \in M$, we have that $A(x)$ is positive definite, but 
\[
\det f[A(x)] = \det \begin{pmatrix}
    f(x) & f(a) \\
    f(a) & 0 
\end{pmatrix} = -f(a)^2 \not\in \F_q^+.  
\]
We therefore conclude that $f(-b) \ne 0$ and so $f(\F_q^-) \subseteq \F_q^-$. 

\medskip

\noindent{\bf Step 2:} $f$ is injective on $\F_q^-$. Suppose $f(-a) = f(-b) =: y$ for some $a, b \in \F_q^+$ with $a \ne b$. Notice that $y \in \F_q^-$ by Step 1. Thus $-y \in \F_q^+$ and so there exists $\alpha \in \F_q^+$ such that $f(\alpha) = -y$. Consider the matrices 
\[
A(x) = \begin{pmatrix}
    x & -a \\
    -a & \alpha
\end{pmatrix}, \qquad B(x) = \begin{pmatrix}
    x & -b \\
    -b & \alpha
\end{pmatrix}.
\]
Let 
\begin{align*}
M_A := \{x \in \F_q^+ : \alpha x - a^2 \in \F_q^+\}, \qquad 
M_B := \{x \in \F_q^+ : \alpha x - b^2 \in \F_q^+\}. 
\end{align*}
Clearly, $A(x)$ is positive definite if and only if $x \in M_A$, and $B(x)$ is positive definite if and only if $x \in M_B$. Also, $\det f[A(x)] = \det f[B(x)] = -y (f(x)+y).$ 
Since $-y \in \F_q^+$, the matrices $f[A(x)]$ and $f[B(x)]$ are positive definite if and only if $x \in \F_q^+$ and $f(x)+y \in \F_q^+$. Using Lemma \ref{Ldoublesquare}, 
\[
|\{x \in \F_q^+ : f(x) + y \in \F_q^+\}| = \frac{q-3}{4}.
\]
We will now prove that $|M_A \cup M_B| > \frac{q-3}{4}$. First, notice that 
\[
x \in M_A \iff x,x - a^2\alpha^{-1} \in \F_q^+. 
\]
Thus, by Lemma \ref{Ldoublesquare}, we have $|M_A| = \frac{q-3}{4}$. Similarly, $|M_B| = \frac{q-3}{4}$. To prove that $|M_A \cup M_B| > \frac{q-3}{4}$, it therefore suffices to show $|M_A \cap M_B| < \frac{q-3}{4}$. Let $s := a^2 \alpha^{-1}$ and $t := b^2 \alpha^{-1}$. Then $|M_A \cap M_B|$ counts the number of $x \in \F_q$ such that $x \in \F_q^+$, $x-s \in \F_q^+$, and $x-t \in \F_q^+$. By Lemma \ref{lem:triple}, we have $|M_A \cap M_B| < \frac{q-3}{4}$ whenever $q\geq 11$. The $q=3$ case was already addressed at the beginning of the proof so the only case left is when $q=7$. In that case, $\F_7^+ = \{1,2,4\}$ and $x, s, t \in \F_7^+$ must be distinct. Examining all $6$ possibilities, we always have $x-s \not\in \F_7^+$ or $x-t \not\in \F_7^+$. It follows that $|M_A \cap M_B| = 0$ and the argument holds for $q=7$ as well.

This proves $|M_A \cup M_B| > \frac{q-3}{4}$. As a consequence, there exists $x^* \in M_A \cup M_B$ such that $f(x^*) + y \not\in \F_q^+$. For such an $x^*$ we have either $A(x^*)$ is positive definite, but $f[(A(x^*)]$ is not; or $B(x^*)$ is positive definite, but $f[B(x^*)]$ is not. This contradicts our assumption and therefore proves that $f$ is bijective on $\F_q^-$. This concludes the proof. 
\end{proof}

We next show a positivity preserver $f$ over $M_2(\F_q)$ must be an odd function, and $f(x^2) = f(x)^2$ for all $x\in \F_q$.

\begin{lemma}\label{L_odd_and_square}
Let $\F_q$ be a finite field with $q \equiv 3 \pmod 4$. Suppose $f: \F_q \to \F_q$ preserves positivity on $M_2(\F_q)$ and $f(1) = 1$. Then 
\begin{enumerate}
    \item $f(-x)=-f(x)$ for all $x\in \F_q$, and
    \item $f(x^2) = f(x)^2$ for all $x \in \F_q$.    
\end{enumerate}
\end{lemma}
\begin{proof}
By Lemma \ref{lbijective}, $f$ is bijective on $\F_q$ and $f$ is bijective on $\F_q^+$. We use similar ideas to prove both of the statements.
\begin{enumerate}
    \item  Fix $x \in \F_q^+$. To show $f(-x)=-f(x)$, it is suffices to show that $f(x)^2=f(-x)^2$. Let $z=f(-x)^2/f(x)$ and let $y$ be the preimage of $z$ under $f$. Then we know that both $y$ and $z$ are in $\F_q^+$. Note that if $x=y$, then we are done since $z=f(y)=f(x)$ implies that $f(x)^2=f(-x)^2$. Next assume that $x \neq y$. Consider $
A = \begin{pmatrix}
y & -x \\
-x & x 
\end{pmatrix}.$ By definition, $\det f[A]=zf(x)-f(-x)^2=0$. Thus, $A$ is not positive definite, that is, $\det A=x(y-x) \in \F_q^- \cup \{0\}$. Since $x \in \F_q^+$ and $x\neq y$, we must have $x-y \in \F_q^+$. Next, consider the matrix 
$
B = \begin{pmatrix}
x & y \\
y & y
\end{pmatrix}.
$ The matrix $B$ is positive definite since $x \in \F_q^+$ and $\det B=y(x-y)\in \F_q^+$. Thus, $f[B]$ is also positive definite. In particular, $\det f[B]=z(f(x)-z)\in \F_q^+$. It follows that $f(x)^2-f(-x)^2 \in \F_q^+$. Finally, consider 
$
C = \begin{pmatrix}
x & -x \\
-x & x 
\end{pmatrix}.
$ The matrix $C$ is singular, while $f[C]$ is positive definite. But since $f$ is bijective on $\F_q$, its entrywise action on $M_2(\F_q)$ is also bijective and maps the set of positive definite matrices to itself. As a consequence, the singular matrix $C$ cannot be mapped to a positive definite matrix by $f$, a contradiction.

\item In view of (1), it suffices to prove the result for $x \in \F_q^+$. Fix $x \in \F_q^+$, let $z=f(x)^2/f(x^2)$, and let $y$ be the preimage of $z$ under $f$. Then we know that both $y$ and $z$ are in $\F_q^+$. If $y=1$, then we are done. Assume $y \neq 1$ and consider 
$
A = \begin{pmatrix}
x^2 & x \\
x & y 
\end{pmatrix}.
$ By definition, $\det f[A]=f(x^2)z-f(x)^2=0$. Thus, $A$ is not positive definite, that is, $\det A=x^2(y-1) \in \F_q^- \cup \{0\}$. Since $y \neq 1$, we must have $1-y \in \F_q^+$. Next, consider the matrix 
$
B = \begin{pmatrix}
1 & y \\
y & y
\end{pmatrix}.
$ The matrix $B$ is positive definite since $\det B=y(1-y)\in \F_q^+$. Thus, $f[B]$ is also positive definite. In particular, $\det f[B]=z(1-z)\in \F_q^+$ and thus $1-z \in \F_q^+$. It follows that $f(x^2)-f(x)^2 \in \F_q^+$. Finally, consider 
$
C = \begin{pmatrix}
x^2 & x \\
x & 1 
\end{pmatrix}.
$ The matrix $C$ is singular, while $f[C]$ is positive definite, a contradiction. \qedhere
\end{enumerate}
\end{proof}

With the previous two preliminary results in hand, we can now prove the main result of this section, which immediately implies Theorem \ref{ThmB}. 

\begin{theorem}\label{ThmPres3mod4}
    Let $\F_q$ be a finite field with $q \equiv 3 \pmod 4$ and let $f: \F_q \to \F_q$ be such that $f$ preserves positivity on $M_2(\F_q)$, and $f(1)=1.$ Then $f(x) = x^{p^\ell}$ for some $\ell = 0,1,\ldots,k-1$. 
\end{theorem}
\begin{proof} 
By Theorem~\ref{Tcarlitz}, it suffices show that $\eta(a-b) = \eta(f(a)-f(b))$ for all $a,b \in \F_q$. This is clear when $a=0$ or $b=0$ since by Lemma~\ref{lbijective}, we have $\eta(c) = \eta(f(c))$ for all $c\in \F_q$. Also, notice that if $\eta(a-b) = -1$, then $\eta(b-a) = 1$ and $\eta(f(a)-f(b)) = -\eta(f(b) - f(a))$. Thus, it suffices to show that if $\eta(a-b)=1$, then $\eta(f(a)-f(b))=1$. We consider the following three cases. In the following discussion we use the fact that $f(-a) = -f(a)$ for all $a\in \F_q$ from Lemma~\ref{L_odd_and_square} (1).

\noindent{\bf Case 1:} $\eta(a) = \pm1$ and $\eta(b) = 1$. Consider the positive definite matrix $A = \begin{pmatrix}
            b & b\\
            b & a
        \end{pmatrix}$. Then $f[A] = \begin{pmatrix}
            f(b) & f(b)\\
            f(b) & f(a)
        \end{pmatrix}$ is also positive definite, which implies that $\eta(f(a)-f(b)) = 1$.

\noindent{\bf Case 2:} $\eta(a) = -1$ and $\eta(b) = -1$. Consider the positive definite matrix $A = \begin{pmatrix}
            -a & -a\\
            -a & -b
        \end{pmatrix}$. Since $f$ is odd,  $f[A] =  \begin{pmatrix}
            -f(a) & -f(a)\\
            -f(a) & -f(b)
        \end{pmatrix}$ is also positive definite, which implies that $\eta(f(a)-f(b)) = 1$.

\noindent{\bf Case 3:} $\eta(a) = 1$ and $\eta(b) = -1$. Here we use Lemma~\ref{L_odd_and_square} (2) which asserts that $f$ satisfies $f(x^2)=f(x)^2$ for all $x\in \F_q.$ Now, consider $a+b$. If $b = -a$, then $1 = \eta(a-b) = \eta(2a) = \eta(a)\eta(2) = \eta(2)$. Hence, since $f$ is odd, we get $$\eta(f(a)-f(b)) = \eta(f(a)-f(-a)) = \eta(2f(a)) = \eta(2) \eta(f(a)) = \eta(2) = 1.$$ If in addition $\eta(a+b) = 1$, then $\eta(a^2-b^2) = \eta((a+b)(a-b)) = 1$. By using Case $1$ we have $1 = \eta(f(a) - f(-b)) = \eta(f(a)+f(b))$ and $1 = \eta(f(a^2)-f(b^2)) = \eta(f(a)^2-f(b)^2)$. Thus, $\eta(f(a)-f(b)) = 1$. Lastly, if $\eta(-a-b) = 1$, then $\eta(b^2-a^2) = \eta((-a-b)(a-b)) = 1$. By using cases $1$ and $2$ we have $1 = \eta(f(-a)-f(b)) = \eta(-f(a)-f(b))$ and $$1 = \eta(f(b^2)-f(a^2)) = \eta(f(b)^2-f(a)^2) = \eta((-f(a)-f(b))(f(a)-f(b))).$$ Thus, $\eta(f(a)-f(b)) = 1$.
\end{proof}

With the above results in hand, we can now prove Theorem \ref{ThmB}. 
\begin{proof}[Proof of Theorem \ref{ThmB}]
Using Lemma~\ref{lbijective}, we assume without loss of generality that $f(1) = 1$. Suppose $(4)$ holds. Using the fact that $(a+b)^{p^\ell} = a^{p^\ell} + b^{p^\ell}$ for all $a, b \in \F_q$, we have 
\[
\eta(a^{p^\ell} - b^{p^\ell}) = \eta((a-b)^{p^\ell}) = \eta(a-b)^{p^\ell} = \eta(a-b).
\]
This proves $(4) \implies (3)$. The converse implication is Theorem~\ref{Tcarlitz}. Thus $(3) \iff (4)$. 

That $(4) \implies (2)$ follows from Proposition \ref{PFrobenius} and $(2) \implies (1)$ is trivial. To prove $(1) \implies (4)$, it suffices to assume $n=2$. If $n > 2$, then one can embed any $2 \times 2$ positive definite matrix $A$ into $M_n(\F_q)$ using a block matrix $A \oplus I_{n-2}$, where $I_{n-2}$ denotes the $(n-2)$-dimensional identity matrix. We therefore assume that $n=2$ and the result follows by Theorem \ref{ThmPres3mod4}.
\end{proof}

As explained at the beginning of Section \ref{S3mod4}, the $(1) \implies (4)$ implication of Theorem \ref{ThmB} is easier to prove under the assumption that $f$ preserves positivity on $M_3(\F_q)$. In that case, the larger test set of $3 \times 3$ matrices makes it easier to deduce the properties of the preservers. We therefore provide a simpler proof of Theorem \ref{ThmB} below under the assumption that $n \geq 3$ in (1) and (2). The proof avoids the use of Weil's bound, as well as Lemma \ref{L_odd_and_square} and Theorem \ref{ThmPres3mod4}. 

\begin{theorem}[Special Case of Theorem \ref{ThmB} for $n \geq 3$]\label{ThmBforn3}
Let $q \equiv 3 \pmod 4$ and let $f: \F_q \to \F_q$. Then the following are equivalent: 
\begin{enumerate}
\item $f$ preserves positivity on $M_n(\F_q)$ for some $n \geq 3$. 
\item $f$ preserves positivity on $M_n(\F_q)$ for all $n \geq 3$. 
\item $f(0) = 0$ and $\eta(f(a)-f(b)) = \eta(a-b)$ for all $a, b \in \F_q$.
\item $f$ is a positive multiple of a field automorphism of $\F_q$, i.e., there exist $c \in \F_q^+$ and $0 \leq \ell \leq k-1$ such that $f(x) = cx^{p^\ell}$ for all $x \in \F_q$. 
\end{enumerate}
\end{theorem}

\begin{proof}
We only prove $(1) \implies (3)$ since the other implications are proved as in the proof of Theorem \ref{ThmB}. Without loss of generality, we assume $f(1)=1$.  Suppose $(1)$ holds. Without loss of generality, we can assume $n=3$ (the general case follows by embedding $3 \times 3$ positive definite matrices into larger matrices of the form $A \oplus I_{n-3}$). By Lemma \ref{L:Charq} (2) we have $f(0) = 0$. 

If $\eta(a-b) = 0$, then we are done. Let us assume that $\eta(a-b) = 1$ and consider the following three cases. 
\begin{description}
\item[Case 1] Assume $b = 0$. Then $\eta(a) = 1$, and therefore by using Lemma \ref{L:Charq} (1) we have $\eta(f(a) - f(0)) = \eta(f(a)) = 1$. 

\item[Case 2] Assume $\eta(b) = 1$. Then the matrix 
\begin{align*}
    A = \begin{pmatrix}
        b & b & 0\\
        b & a & 0\\
        0 & 0 & 1
    \end{pmatrix}
\end{align*}
is positive definite. Hence, under the map $f$, we have $\det f[A] = f(b) (f(a)-f(b)) \in \F_q^+$. Thus, $\eta(f(a)-f(b)) = 1$ since $\eta(f(b)) = 1$.
\item[Case 3] Assume $\eta(b) = -1$. Consider the linear map $g: \F_q \to \F_q$ given by $g(x) = x+b$. Note that $g$ is bijective, $g(0) = b$ and $g(-b) = 0$. Thus, there must exist $x_0 \in \F_q$ such that $\eta(x_0) = -1$ and $\eta(g(x_0)) = 1$. Let $x_0 = -c$ where $\eta(c) = 1$, and hence $\eta(b-c) =1$. Thus, the matrix
\begin{align*}
    A = \begin{pmatrix}
        c & c & c\\
        c & b & b\\
        c & b & a
    \end{pmatrix}
\end{align*}
is positive definite. Hence, under the map $f$, we have $\det f[A] = f(c)(f(b)-f(c))(f(a)-f(b))$. We know that $\eta(f(c)) = 1$, and using the previous case applied with $a' = b$ and $b' = c$, we conclude that $\eta(f(b)-f(c)) = 1$. Thus, $\eta(f(a)-f(b)) = 1$. 
\end{description}
On the other hand, if $\eta(a-b) = -1$, then $\eta(b-a) = 1$. Hence, by the above argument $\eta(f(b) - f(a)) = 1$. That implies $\eta(f(a)-f(b)) = -1$. Thus, $(1) \implies (3)$ and the result follows.
\end{proof}

\section{Odd characteristic: $q \equiv 1 \pmod 4$--Reductions to injectivity on $\F_q^+$}\label{S1mod4}

Throughout the next two sections, we assume $q \equiv 1 \pmod 4$ is a prime power. We adopt the combinatorial viewpoint of identifying the elements of $\F_q$ with the vertices of the Paley graph $P(q)$; see Section~\ref{SSPaley} for basic properties of Paley graphs.

The main purpose of the section is to prove Proposition~\ref{prop:main}, namely, showing that injectivity of a preserver $f$ on $\F_q^+$ together with $f(1)=1$ force $f$ to be a field automorphism. We also discuss how Proposition~\ref{prop:main} leads to the characterization of positivity preservers over $M_3(\F_q)$. In a similar spirit as in Section~\ref{S3mod4}, we also present an alternative simpler proof of this result at the end of the section.

\subsection{Paley graphs}\label{SSPaley}
Paley graphs have been well-studied in the literature. We begin by recalling their definition and some of their basic properties.
\begin{definition}\label{Def:Paley}
The Paley graph $P(q)$ is the graph whose vertices are the elements of $\F_q$ and where two vertices $a, b \in \F_q$ are adjacent if and only $a-b \in \F_q^+$.
\end{definition}

\noindent Given a graph $G = (V, E)$ and a vertex $v \in V$, we denote the set of vertices adjacent to $v$ (i.e., the neighborhood of $v$) by $N(v)$. 
\begin{lemma}[{\cite[Proposition 9.1.1]{brouwer2011spectra}}]\label{Lemma:Paley_is_SRG}
The Paley graph $P(q)$ is a strongly regular graph with parameters $(q, \frac{q-1}{2}, \frac{q-5}{4}, \frac{q-1}{4})$. In other words,
 \begin{enumerate}
    \item[$(1)$] For any vertex $v$, we have $|N(v)| = \frac{q-1}{2}$.
    \item[$(2)$] For any two adjacent vertices $u,v$, we have  $|N(u)\cap N(v)| =\frac{q-5}{4}$.
    \item[$(3)$] For any two non-adjacent vertices $u,v$, we have $|N(u) \cap N(v)| = \frac{q-1}{4}$. 
\end{enumerate}    
\end{lemma}

Let $\Gamma(q)$ be the subgraph of $P(q)$ induced by $\F_q^+$. 
Muzychuk and Kov\'acs \cite{muzychuk2005solution} confirmed a conjecture of Brouwer on the automorphisms of $\Gamma(q)$.
\begin{theorem}[\cite{muzychuk2005solution}]\label{thm:localPaley}
Let $p$ be a prime and $q=p^k \equiv 1 \pmod 4$. The automorphisms of the graph $\Gamma(q)$ are precisely given by the maps $x \mapsto ax^{\pm p^l}$, where $a \in \F_q^+$ and $l \in \{0,1,\ldots,k-1\}$.
\end{theorem}

The following corollaries follow immediately from Lemma~\ref{Lemma:Paley_is_SRG}.

\begin{corollary}\label{cor:commonnbd}
For each $x \in \F_q^+$, the number of $y \in \F_q^+$ such that $x-y \in \F_q^+$ is $\frac{q-5}{4}$. In particular, $\Gamma(q)$ is a regular graph.   
\end{corollary}
\begin{proof}
The required number of elements is precisely $|N(0) \cap N(x)|$ with $0$ and $x$ adjacent. 
\end{proof}

\begin{corollary}\label{cor:cab}
Let $a,b \in \F_q^+$ such that $a \neq b$. Then there is $c \in \F_q^-$, such that $a-c \in \F_q^+$ and $b-c \in \F_q^-$. 
\end{corollary}
\begin{proof}
We prove this lemma by considering the neighborhood of $a$ and $b$ in $P(q)$. Note that $$|N(a) \cap \F_q^-|=|N(a)|-|N(a) \cap N(0)|-1=\frac{q-1}{2}-\frac{q-5}{4}-1=\frac{q-1}{4}.$$ Similarly, $|N(b) \cap \F_q^-|=\frac{q-1}{4}$. On the other hand, since $0 \in N(a)\cap N(b)$, we have $|N(a)\cap N(b) \cap \F_q^*|\leq \frac{q-1}{4}-1 = \frac{q-5}{4} < \frac{q-1}{4}$. In particular, $|N(a) \cap N(b) \cap \F_q^-| < \frac{q-1}{4}$. Thus, the sets $N(a) \cap \F_q^-$ and $N(b) \cap \F_q^-$ have the same size but are not the same. This implies the existence of the desired $c$.   
\end{proof}

\begin{corollary}\label{cor:cab2}
Let $a,b \in \F_q^-$ such that $a \neq b$. Then there is $c \in \F_q^+$, such that $a-c \in \F_q^-$ and $b-c \in \F_q^+$.
\end{corollary}
\begin{proof}
Let $x \in \F_q^-$, and set $a'=ax$ and $b'=bx$. Applying Corollary~\ref{cor:cab} to $a'$ and $b'$, we can find $c' \in \F_q^-$ such that $a'-c' \in \F_q^+$ and $b'-c' \in \F_q^-$. Then $c=c'/x$ is as desired.   
\end{proof}

\noindent Finally, we combine Lemma \ref{Lemma:Paley_is_SRG} and Lemma~\ref{lem:triple} to deduce the following corollary.

\begin{corollary}\label{cor:triple2}
Let $q \geq 13$. Let $a,b\in \F_q^+$ with $a \neq b$. Then $N(0)\cap N(a)\neq N(0)\cap N(b)$.   
\end{corollary}
\begin{proof}
When $q=13$, directly examining the $15$ possible pairs $(a,b)$ yields the result. Let us now assume $q \geq 17$. Assume otherwise that $N(0)\cap N(a)=N(0)\cap N(b)$. It follows that $|N(a)\cap N(b)\cap N(0)| = |N(a) \cap N(0)| = \frac{q-5}{4}$ by Lemma \ref{Lemma:Paley_is_SRG}. However, this contradicts Lemma~\ref{lem:triple} which states that $|N(0) \cap N(a) \cap N(b)|<\frac{q-5}{4}$. 
\end{proof}

\subsection{A sufficient condition}\label{SSSufficient}
We now prove that to show that positivity preservers on $M_2(\F_q)$ are positive multiples of field automorphisms, it suffices to show injectivity on $\F_q^+$. 

\begin{proposition}\label{prop:main}
Let $q=p^k$ be a prime power with $q \equiv 1 \pmod 4$ and let $f$ be a positivity preserver over $M_2(\F_q)$ with $f(1)=1$. Assume additionally that $f$ is injective on $\F_q^+$. Then there exists $0\leq j\leq k-1$ such that 
$f(x)=x^{p^j}$ for all $x \in \F_q$.    
\end{proposition}

\noindent We first prove the following two lemmas.

\begin{lemma}\label{lem:local}
Let $q$ be a prime power with $q \equiv 1 \pmod 4$ and let $f$ be a positivity preserver over $M_2(\F_q)$ with $f(1)=1$. If $f$ is injective on $\F_q^+$, then $f(0)=0$, and $f$ (restricted to $\F_q^+$) is an automorphism of $\Gamma(q)$.
\end{lemma}
\begin{proof}
By Lemma~\ref{lem:+}, $f(\F_q^+) \subset \F_q^+$. Since $f$ is injective on $\F_q^+$, $f$ is bijective on $\F_q^+$. For the sake of contradiction, assume that $f(0)\neq 0$. Let $x \in \F_q^+$, and consider the matrix 
$$
A=
\begin{pmatrix} 
x&1\\
1&0
\end{pmatrix}.
$$    
Clearly, $A$ is positive definite. Under the map $f$, we have 
$$f(x)f(0)-f(1)^2=f(x)f(0)-1 \in \F_q^+$$
for all $x \in \F_q^+$. Since $f$ is bijective over $\F_q^+$, as $x$ runs over $\F_q^+$, $f(x)$ also runs over $\F_q^+$. Equivalently, for any $y \in \F_q^+$, we have $f(0)y-1 \in \F_q^+$. If $f(0) \in \F_q^+$, then this implies $\F_q^+ \subseteq N(1)$ and so $N(1) = \F_q^+$ since $P(q)$ is $\frac{q-1}{2}$-regular. Similarly, if $f(0) \in \F_q^-$, then $N(1) = \F_q^-$. Both cases contradict the fact that $|N(1) \cap N(0)| = \frac{q-5}{4}$ (Lemma \ref{Lemma:Paley_is_SRG}).    

Next consider the graph $\Gamma(q)$. We need to show that $f$ (restricted to $\F_q^+$) is an automorphism of $\Gamma(q)$. Let $x \in \F_q^+$. By Lemma~\ref{lem:sq}, if $y \in \F_q^+$ such that $x-y \in \F_q^+$, then we also have $f(x)-f(y)\in \F_q^+$. Since $x, f(x) \in \F_q^+$ and they have the same number of neighbors (Corollary~\ref{cor:commonnbd}), the neighborhood of $f(x)$ must be precisely the image of the neighborhood of $x$ under the map $f$. This completes the proof.
\end{proof}

\begin{lemma}\label{LneighIncl}
Let $q \equiv 1 \pmod 4$ and let $f$ be a positivity preserver on $M_2(\F_q)$. Then, for $a, b \in \F_q^+$ with $a \ne b$, we have
\[
f(N(a) \cap N(b)) \subseteq N(f(a)) \cap N(f(b)).
\]
\end{lemma}
\begin{proof}
For $a \in \F_q^+$ and $x \in \F_q$, consider the matrix
\[
A = \begin{pmatrix}
a & a \\
a & x
\end{pmatrix}.
\]
Then $A$ is positive definite if and only if $x \in N(a)$. By Lemma \ref{lem:+}, the matrix $f[A]$ is positive definite if and only if $f(x) \in N(f(a))$. It follows that $f(N(a)) \subseteq N(f(a))$. The result immediately follows by taking intersections. 
\end{proof}

Now we are ready to present the proof of Proposition~\ref{prop:main}.

\begin{proof}[Proof of Proposition~\ref{prop:main}]
By Lemma~\ref{lem:local}, $f(0)=0$ and $f$ (restricted to $\F_q^+$) is an automorphism of $\Gamma(q)$. Since $f(1)=1$, Theorem~\ref{thm:localPaley} implies that there is $0 \leq j\leq k-1$, such that $f(x)=x^{p^j}$ for all $x \in \F_q^+$, or $f(x)=x^{-p^j}$ for all $x \in \F_q^+$. We address the cases where $q=5$, $q=9$, and $q \geq 13$ separately. 

\noindent {\bf Case 1: $q=5$.} Note that $\F_5^+= \{1,4\} = \{-1,1\}$ and we must have either $f(x) = x$ or $f(x) = x^{-1}$ for all $x \in \F_5^+$. In both cases, we obtain $f(4) = 4$. Consider the matrix
\[
A = \begin{pmatrix}
1 & 1 \\
1 & 2 
\end{pmatrix}.
\]
Since $A$ is positive definite, $f(2) -1 \in \F_q^+$ and so $f(2) \in \{0,2\}$. Using Lemma \ref{lem:nonzero}, we obtain $f(2) = 2$. Finally, consider the positive definite matrix 
\[
B = \begin{pmatrix}
1 & 2 \\
2 & 3
\end{pmatrix}, 
\]
we conclude that $f(3) - 4 \in \F_q^+$, i.e., $f(3) \in \{0,3\}$. As above, we conclude $f(3) = 3$. This shows $f(x) = x$ for all $x \in \F_5$.
\medskip

\noindent{\bf Case 2. $q = 9$.} We identify $\F_9$ with $\F_3 + i \F_3$ where $i^2 = -1$. We have $\F_9^+ = \{1,-1,i, -i\}$ and we have either $f(x) = x, x^{-1}, x^3, x^{-3}$ for all $x \in \F_9^+$. In all cases, we have $f(2) = 2$. If $f(x) = x$ or $f(x) = x^{-3}$, we have $f(i) = i$ and $f(-i) = -i$. If $f(x) = x^{-1}$ or $f(x) = x^3$, we have $f(i) = -i$ and $f(-i) = i$. We consider two subcases. 

\noindent {\bf Case 2a: $f(i) = i$ and $f(-i) = -i$.} In this case, we have $f(x) = x$ for all $x \in \F_9^+$. Lemma \ref{LneighIncl} then shows that, for $a, b \in \F_9^+$ with $a \ne b$,  we have $f(N(a) \cap N(b)) \subseteq N(a) \cap N(b)$. Observe that
\begin{align*}
\{0,1+i\} &= N(1) \cap N(i), \quad&& \{0,1-i\}= N(1) \cap N(-i), \\
\{0,-1+i\} &= N(-1) \cap N(i), \quad&&\{0, -1-i\}= N(-1) \cap N(-i). 
\end{align*}
We conclude from Lemma \ref{lem:nonzero} that $f(\pm 1 \pm i) = \pm 1 \pm i$ and therefore $f(x) = x$ for all $x \in \F_9$. 

\noindent {\bf Case 2b: $f(i) = -i$ and $f(-i) = i$}. We will show that this implies $f(x) = x^3$ for all $x \in \F_9$. First observe that we already have $f(0) = 0 = 0^3, f(1) = 1 = 1^3, f(2) = 2 = 2^3, f(i) = -i = i^3, f(-i) = i = (-i)^3$. Next, using Lemma \ref{LneighIncl}, we obtain
\begin{align*}
f(1+i) \in f(N(1) \cap N(i)) \subseteq N(f(1)) \cap N(f(i)) = N(1) \cap N(-i) = \{0, 1-i\}.
\end{align*}
Hence, by Lemma \ref{lem:nonzero}, $f(1+i) = 1-i = (1+i)^3$. Similarly, it follows that $f(1-i) = 1+i = (1-i)^3, f(-1+i)=-1-i = (-1+i)^3, f(-1-i) = -1+i = (-1-i)^3$. This proves $f(x) = x^3$ for all $x \in \F_9$.
 \medskip

\noindent{\bf Case 3:} We now assume $q \geq 13$. Let $0 \leq j \leq k-1$, and let $e \in \{p^j,-p^j\}$ such that $f(x)=x^e$ for all $x \in \F_q^+$. We study the values of $f$ on $\F_q^-$. Define the following matrices:
\begin{align*}
A(x,y):=\begin{pmatrix} 1 & y \\ y & x \end{pmatrix} \quad \mbox{for} \quad x,y\in \F_q.
\end{align*}
Let $y \in \F_q^-$ be fixed. Then for each $x \in \F_q^+$ such that $x-y^2 \in \F_q^+$, the matrix $A(x,y)$ is positive definite. Under the map $f$, the matrices $f[A(x,y)]$ are also positive definite. Thus, $f(1)f(x)-f(y)^2=f(x)-f(y)^2 \in \F_q^+$. By Corollary~\ref{cor:commonnbd}, there are exactly $\frac{q-5}{4}$ many $x$ such that $x \in \F_q^+$ and $x-y^2 \in \F_q^+$; for each such $x$, we also have $f(x) \in \F_q^+$ and $f(x)-f(y)^2 \in \F_q^+$. Since $f$ is injective on $\F_q^+$, it follows that $N(0)\cap N(f(y)^2)=f(N(0)\cap N(y^2))$. Using Corollary~\ref{cor:triple2}, we conclude that $f(y)^2$ is the unique $t \in \F_q^+$ such that $N(0) \cap N(t) = f(N(0) \cap N(y^2))$. It follows that $f(y)^2=(y^2)^{e}$ since $N(0)\cap N(y^{2e})=(N(0)\cap N(y^2))^e$. Therefore, we have shown that for each $y \in \F_q^-$, we have $f(y)^2=(y^2)^{e}$, and thus $f(y)=y^{e}$ or $f(y)=-y^{e}$. 
    
We now claim that $f(y)=y^{e}$ for all $y \in \F_q$. For the sake of contradiction, assume that $f(y)=-y^e$ for some $y \in \F_q^-$. For each $w \in \F_q^*$ such that $y-w^2 \in \F_q^+$, consider the positive definite matrices $A(y,w).$ Then $f[A(y,w)]$ are positive definite. We have two possibilities: $e>0$ and $e<0.$

We first consider the case $e>0.$ Then $$\det f[A(y,w)]=f(1)f(y)-f(w)^2=-y^e-w^{2e}=(-y-w^2)^e \in \F_q^+.$$ It follows that $-y-w^2 \in \F_q^+$. Since $0$ and $y$ are not adjacent, the number of common neighbors of $0$ and $y$ is $\frac{q-1}{4}$; similarly,  the number of common neighbors of $0$ and $-y$ is $\frac{q-1}{4}$. Therefore, the common neighborhood of $0$ and $y$ coincides with the common neighborhood of $0$ and $-y$, contradicting Corollary~\ref{cor:triple2}. We have thus shown that $f(y)=y^e$ for all $y \in \F_q$. This map is indeed a positivity preserver by Proposition~\ref{PFrobenius}.

Next, we consider the case $e<0,$ and set $d=-e>0$. Again $$\det f[A(y,w)]=f(1)f(y)-f(w)^2=-y^e-w^{2e}=-\frac{1}{y^d}-\frac{1}{w^{2d}}=-\frac{y^d+w^{2d}}{y^dw^{2d}}=-\frac{(y+w^2)^d}{(yw^2)^d} \in \F_q^+.$$ 
It follows that $(y+w^2)^d \in \F_q^-$ and thus $-y-w^2 \in \F_q^-$. Therefore, for each $w^2 \in \F_q^+$ such that $y-w^2 \in \F_q^+$, we have $-y-w^2 \in \F_q^-$. In other words, $0,y,-y$ do not have any common neighbor, which contradicts Lemma~\ref{lem:triple} when $q>25$. When $q \in \{13,17,25\}$, we can use a simple computation to verify that $0,y,-y$ do have a common neighbor for each $y \in \F_q^-$. 

We have thus shown that $f(x)=x^e$ for all $x \in \F_q$. We will show that this map is not a positivity preserver when $e < 0$. Note that the number of common neighbors of $0$ and $1$ is $\frac{q-5}{4}$, equivalently, the number of neighbors of $1$ in $\F_q^+$ is $\frac{q-5}{4}$. Since the number of neighbors of $1$ is $\frac{q-1}{2}$, we can pick $y \in \F_q^-$ such that $y-1 \in \F_q^+$. Consider the positive definite matrix $A(y,1).$ Then $f[A(y,1)]$ is positive definite so that $1-y^e \in \F_q^+$. However, note that (for $d=-e$)
$$
1-y^e=1-\frac{1}{y^d}=\frac{y^d-1}{y^d}=\frac{(y-1)^d}{y^d} \in \F_q^-
$$
since $y-1 \in \F_q^+$ and $y \in \F_q^-$, a contradiction.
\end{proof}

\subsection{Applications of Proposition~\ref{prop:main}}\label{SSApplications}

In view of Proposition~\ref{prop:main}, we now examine three sufficient conditions to guarantee that $f$ is injective on $\F_q^+$ and discuss their applications to Theorem~\ref{ThmC} and Theorem~\ref{ThmD}. 

Recall from Section~\ref{Seven} that when $q$ is even, a positivity preserver reduces to a map that preserves non-singularity. This inspires us to prove the following.

\begin{proposition}
Let $q \equiv 1 \pmod 4$ and let $f: \F_q \to \F_q$. If $f$ maps nonsingular matrices to nonsingular matrices, then $f$ is injective on $\F_q^+$.
\end{proposition}
\begin{proof}
Suppose $a,b \in \F_q^+$ with $a \neq b$ and $f(a)=f(b)$. Consider the matrix 
$$
A=
\begin{pmatrix} 
b&b\\
b&a
\end{pmatrix}.
$$
It has determinant $b(a-b)$ and thus it is nonsingular. However, all entries in $f[A]$ are the same. 
\end{proof}

\begin{proposition}\label{prop:sign}
Let $q \equiv 1 \pmod 4$ and let $f$ be a sign preserver on $M_2(\F_q)$. Then $f$ is injective on $\F_q^+$.    
\end{proposition}
\begin{proof}
Without loss of generality, assume $f(1) = 1$. Suppose first $q=5$. We have $\F_5^+ = \{1,4\}$. Suppose for a contradiction that $f(1) = f(4)$. By Lemma \ref{lem:sq}, $f(1) - f(0) = 1 - f(0) \in \F_5^+$, i.e., $f(0) \in N(1) = \{0,2\}$. If $f(0) = 2$, then $\det f[I_2] = 2 \not\in \F_5^+$, a contradiction. Therefore $f(0) = 0$ and Lemma \ref{lem:nonzero} yields $f(x) \ne 0$ for all $x \in \F_q^*$. Using Lemma \ref{lem:sq} again yields $f(2) \in N(1) = \{0,2\}$ and so $f(2) = 2$. Similarly, $f(3) \in N(f(4)) = N(1) = \{0,2\}$. Thus $f(3) = 2$. Now, consider the positive definite matrix 
\[
A = \begin{pmatrix}
1 & 2 \\
2 & 3
\end{pmatrix}.
\]
Using the above, we obtain $\det f[A] = 3 \not\in \F_5^+$, a contradiction. We must therefore have $f(1) \ne f(4)$ and $f$ is injective on $\F_5^+$.

Next, suppose $q = 9$ and identify $\F_9$ with $\F_3 + i\F_3$ where $i^2 = -1$. We have $\F_9^+ = \{1,-1,i,-i\}$. Consider the following $6$ positive definite matrices:
\[
\begin{pmatrix}
1 & 1 \\
1 & -1
\end{pmatrix}, \quad \begin{pmatrix}
i & 1 \\
1 & i
\end{pmatrix},\quad \begin{pmatrix}
-i & 1 \\
1 & -i
\end{pmatrix}, \quad  \begin{pmatrix}
i & -1 \\
-1 & i
\end{pmatrix}, \quad \begin{pmatrix}
-i & -1 \\
-1 & -i
\end{pmatrix}, \quad \begin{pmatrix}
i & i \\
i & -i
\end{pmatrix}.
\]
If $f(a) = f(b)$ for some $a, b \in \F_9^+$ with $a \ne b$, then $f[A]$ is singular for one of the above matrices. We therefore conclude that $f$ is injective on $\F_9^+$.

Finally, let $q \geq 13$. Suppose $a, b \in \F_q^+$ with $a \neq b$ and $f(a)=f(b)$. Note that$|N(0)\cap N(a)|=|N(0)\cap N(b)|=\frac{q-5}{4}$ by Lemma~\ref{Lemma:Paley_is_SRG}. Corollary~\ref{cor:triple2} implies that $N(0)\cap N(a)\neq N(0)\cap N(b)$. Thus, we can find $x\in (N(0)\cap N(a)) \setminus N(b)$, that is, we have $x \in \F_q^+$ such that $a-x \in \F_q^+$ while $b-x \in \F_q^-$. Consider two matrices
$$
A_1=
\begin{pmatrix} 
x&x\\
x&a
\end{pmatrix}, \qquad
A_2=
\begin{pmatrix} 
x&x\\
x&b
\end{pmatrix}.
$$
Note that $A_1$ is positive definite, so $f(a)-f(x) \in \F_q^+$. On the other hand, $A_2$ is not positive definite, so $f(b)-f(x) \not \in \F_q^+$. This is a contradiction since $f(a)=f(b)$. 
\end{proof}

We are now ready to prove Theorem~\ref{ThmD}.
\begin{proof}[Proof of Theorem~\ref{ThmD}]
Let $f$ be a sign preserver on $M_2(\F_q)$. Then in particular, $f$ is a positivity preserver on $M_2(\F_q)$. When $q$ is even, Theorem~\ref{ThmA} implies that $f$ is a bijective monomial and it is straightforward to verify that a bijective monomial is a sign preserver on $M_2(\F_q)$.

Next assume that $q$ is odd. We claim that $f$ is a positive multiple of a field automorphism of $\F_q$. When $q\equiv 3 \pmod 4$, this follows from 
Theorem~\ref{ThmB}; when $q \equiv 1 \pmod 4$, this follows from Proposition~\ref{prop:main} and Proposition~\ref{prop:sign}. Conversely, one can verify that a positive multiple of a field automorphism of $\F_q$ is a sign preserver on $M_2(\F_q)$ using Proposition~\ref{PFrobenius}.
\end{proof}

\begin{proposition}\label{Pqmod1Injective}
Let $q \equiv 1 \pmod 4$ and let $f: \F_q \to \F_q$. If $f$ is a positivity preserver on $M_3(\F_q)$, then $f$ is injective on $\F_q^+$.
\end{proposition}
\begin{proof}
Suppose $a, b \in \F_q^+$ with $a \neq b$ and $f(a)=f(b)$. By Lemma~\ref{lem:sq}, $a-b \in \F_q^-$.  By Corollary~\ref{cor:cab}, there exists $c \in \F_q^-$, such that $a-c \in \F_q^+$ and $b-c \in \F_q^-$. Now, the matrix
\begin{align*}
    A = \begin{pmatrix}
        a & a & a\\
        a & c & b\\
        a & b & b
    \end{pmatrix}
\end{align*}
is positive definite since the leading principal minors $a, a(c-a), a(b-c)(a-b) \in \F_q^+$. Hence, $f[A]$ is also positive definite. In particular, $f(a) \neq f(b)$, a contradiction. 
\end{proof}

We now have all the ingredients to prove the first 4 equivalences in Theorem \ref{ThmC}. The proof of the $q = r^2$ case relies on Theorem \ref{thm:main} whose proof is given in Section \ref{S1mod4square} below.

\begin{proof}[Proof of Theorem \ref{ThmC}]
We assume without loss of generality that $f(1) = 1$. Suppose first that $q \equiv 1 \pmod 4$ is arbitrary and assume (1) holds. Considering matrices of the form $A \oplus I_{n-3} \in M_n(\F_q)$ where $A \in M_3(\F_q)$, it follows immediately that $f$ preserves positivity on $M_3(\F_q)$. Thus, by Proposition \ref{Pqmod1Injective}, the function $f$ is injective on $\F_q^+$. Proposition \ref{prop:main} then implies that $f(x) = x^{p^j}$ for some $0 \leq j \leq k-1$ and so (4) holds. This proves $(1) \implies (4)$. That (4) is equivalent to (3) is Theorem \ref{Tcarlitz}. Next, Proposition \ref{PFrobenius} shows that $(4) \implies (2)$. Finally, $(2) \implies (1)$ is trivial.

Suppose now $q = r^2$ for some odd integer $r$. Then Theorem \ref{thm:main} shows $(1') \implies (4)$. That $(4) \implies (1')$ is again Proposition \ref{PFrobenius}. This concludes the proof of the theorem. 
\end{proof}

Our proof of the first four equivalences in Theorem \ref{ThmC} rely on Proposition \ref{Pqmod1Injective} to first show that the preserver $f$ is injective on $\F_q^+$, and then on Proposition \ref{prop:main} to conclude that $f$ is an automorphism via a careful analysis of the two possible resulting forms for $f$. Our proofs use Weil's bound and Muzychuk and Kov\'acs' characterization of the automorphisms of the graph $\Gamma(q)$ (Theorem \ref{thm:localPaley}). In the same spirit as Theorem \ref{ThmBforn3} in the $q \equiv 3 \pmod 4$ case, we now provide a more direct proof of the $(1) \implies (4)$ implication in Theorem \ref{ThmC} using Theorem \ref{Tcarlitz} instead of Theorem \ref{thm:localPaley}. Note that, in contrast to Theorem \ref{thm:localPaley} whose proof relies on spectral and Schur ring techniques, there are several known short proofs of Theorem \ref{Tcarlitz} (see \cite[Section 9]{jones2020paley}). The proof below thus provides a significant simplification of our previous argument when $n \geq 3$. 

\begin{proof}[Proof of $(1) \implies (4)$ in Theorem \ref{ThmC}]
Suppose $(1)$ holds. As before, it suffices to assume $n=3$ as the general case follows by embedding $3 \times 3$ matrices into $M_n(\F_q)$. Proposition \ref{Pqmod1Injective} and Lemma~\ref{lem:+} imply that $f$ is bijective over $\F_q^+$. Since $f$ is also a positive preserver on $M_2(\F_q)$, Lemma~\ref{lem:local} implies that $f(0)=0$. Now Lemma~\ref{lem:nonzero} implies that $0 \notin f(\F_q^-)$. By Theorem~\ref{Tcarlitz}, it suffices to show $\eta(f(a)-f(b)) = \eta(a-b)$ for all $a, b \in \F_q^*$. If $a,b \in \F_q^+$, then the statement follows from Lemma~\ref{lem:local}. So we assume that $a \in \F_q^-$ and $b \in \F_q^*$ with $a\neq b$ without loss of generality. We consider the following three cases.

\noindent{\bf Case 1:} $\eta(a-b) = 1$. If $\eta(b)=1$, then Lemma~\ref{lem:sq} implies that $\eta(f(a)-f(b))= 1$. Now, suppose that $\eta(a) = \eta(b) = -1$. By Lemma~\ref{Lemma:Paley_is_SRG}, $N(0) \cap N(a)\neq \emptyset$, thus we can pick $c\in \F_q$ such that $\eta(c)=1$ and $\eta(a-c)=1.$ Thus, the matrix
\begin{align*}
    A = \begin{pmatrix}
        c & c & c\\
        c & a & a\\
        c & a & b
    \end{pmatrix}
\end{align*}
is positive definite since the leading principal minors $c, c(a-c), c(a-c)(b-a)\in \F_q^+$. Hence, under the map $f$, we have $\det f[A] = f(c)(f(a)-f(c))(f(b)-f(a)) \in \F_q^+$. We have $\eta(f(c)) = 1$ and $\eta(f(a)-f(c)) = 1$ by the previous case. Hence, $\eta(f(a)-f(b)) = \eta(f(b)-f(a)) = 1$.  

\noindent{\bf Case 2:} $\eta(a-b) = -1$ and $\eta(b)=1$. By Lemma~\ref{Lemma:Paley_is_SRG}, $N(0) \cap N(a)\neq \emptyset$, thus we can pick $c\in \F_q$ such that $\eta(c)=1$ and $\eta(a-c)=1.$ Then the matrix
\[
A=\begin{pmatrix}
    b & a & a \\
    a & a & a \\
    a & a & c
\end{pmatrix}
\]
is positive definite since all its leading principal minors $b, a(b-a), a(c-a)(b-a) \in \F_{q}^{+}.$ Under the map $f$, we have $\det f[A]=f(a)(f(c)-f(a))(f(b)-f(a)) \in \F_{q}^+.$ By Case~1 above, $\eta(f(c)-f(a))=1$.  Therefore, $\eta(f(a)-f(b))=\eta(f(b)-f(a))=\eta(f(a))=-1.$

\noindent{\bf Case 3:} $\eta(a-b) = -1$ and $\eta(b)=-1$. By Corollary~\ref{cor:cab2}, there exists $c\in \F_q$ with $\eta(c)=1$ such that $\eta(a-c)=-1$ and $\eta(b-c)=1.$ Now the matrix
\[
A=\begin{pmatrix}
    c & c & c \\
    c & b & a \\
    c & a & a
\end{pmatrix}
\]
is positive definite since its leading principal minors $c,c(b-c), c(a-c)(b-a) \in \F_{q}^{+}.$ Thus, under the map $f$, we have 
$\det f[A]=f(c)(f(a)-f(c))(f(b)-f(a)) \in \F_{q}^{+}.$ By Case~2 above, $\eta(f(a)-f(c))=-1.$ Therefore $\eta(f(b)-f(a))=-1.$
\end{proof}

\section{Odd characteristic: $q \equiv 1 \pmod 4$ and $q$ is a square}\label{S1mod4square}

We now address the case where $q = r^2$ for some odd integer $r$. The proof of our characterization is broken up into several subsections. Section \ref{SSPaleySquare} reviews the Erd\H{o}s-Ko-Rado theorem for Paley graphs and provides several important properties of Paley graphs of square order. Section \ref{SSOutlineSquare} provides an outline of our approach. Section \ref{SSInjectiveSquare} concludes the proof of Theorem \ref{ThmC} by proving the injectivity on $\F_q^+$ of positivity preservers on $M_2(\F_q)$. 

Throughout the section, we crucially use the fact that $\F_r^* \subset \F_q^+$. Indeed, let $g$ be a generator of $\F_q^*$; then $g^{r+1}$ is a generator of the subgroup $\F_r^*$ of $\F_q^*$. Since $r+1$ is even, it follows that $g^{r+1}\in \F_q^+$ and thus $\F_r^* \subset \F_q^+$.

\subsection{Paley graphs of square order}\label{SSPaleySquare}

One additional ingredient in our characterization of positivity preservers on $M_2(\F_q)$ with $q=r^2$ is the characterization of maximum cliques in the Paley graph $P(q)$, also known as the Erd\H{o}s-Ko-Rado (EKR) theorem \cite{EKR61} for Paley graphs of square order \cite[Section 5.9]{godsil2016erdos}. Analogous versions of the EKR theorem have been proved in many different combinatorial/algebraic settings; we refer to the book of Godsil and Meagher \cite{godsil2016erdos} for a comprehensive discussion. 

Set $q=r^2$, where $r$ is an odd prime power. Notice that $\F_r$ is a subfield of $\F_q$. A \emph{square translate} of $\F_r$ has the form $\alpha \F_r+\beta$, where $\alpha \in \F_q^+$ and $\beta \in \F_q$. It is easy to verify that square translates of $\F_r$ are cliques in $P(q)$. The EKR theorem for Paley graphs (first proved by Blokhuis \cite{Blo84}; see also \cite{AY22} for a generalization) shows that these are precisely the maximum cliques in $P(q)$. 

\begin{theorem}[\cite{Blo84}]\label{thm:EKR}
In the Paley graph $P(q)$, the clique number of $P(q)$ is $r$. Moreover, all maximum cliques are given by squares translates of the subfield $\F_r$.   
\end{theorem}

Note that $\F_q^*/\F_r^*$ is a well-defined group. One can thus write $\F_q^*$ as a disjoint union of $\F_r^*$-cosets. We say such a coset is a \emph{square coset} if it has the form $a \F_r^*$, where $a$ is a non-zero square in $\F_q$. Theorem~\ref{thm:EKR} implies the following corollary.

\begin{corollary}\label{cor:sqclique}
Let $C \subset \F_q^+$ be a clique in $P(q)$. Then $|C|\leq r-1$ and equality holds if and only if $C$ is a square coset. 
\end{corollary}
\begin{proof}
Since $C \subset \F_q^+$, it is clear that $C \cup \{0\}$ is also a clique. Thus we have $|C|\leq r-1$, with equality if and only if $C \cup \{0\}=\alpha \F_r+\beta$ for some $\alpha, \beta \in \F_q$ and $\alpha \in \F_q^+$ by Theorem~\ref{thm:EKR}. Now, if $0 \in \alpha \F_r+\beta$, then there exists $x \in \F_r$ such that $\alpha x+\beta=0$, and it follows that $C \cup \{0\}=\alpha \F_r+\beta=\alpha(\F_r-x)=\alpha \F_r$. Thus, $C=\alpha \F_r^*$ is a square coset.
\end{proof}

Next, we collect several required properties of Paley graphs of square order which are needed in our proof. The following lemma is well-known; we include a short proof for completeness.

\begin{lemma}\label{lem:Hoffman}
Let $u \in \F_q$ and let $C$ be a square coset such that $u \notin C$. If $u \in \F_q^+$, then the number of neighbors of $u$ in $C$ is exactly $\frac{r-3}{2}$; if $u \in \F_q^-$, then the number of neighbors of $u$ in $C$ is exactly $\frac{r-1}{2}.$ 
\end{lemma}
\begin{proof}
We use the fact that the Paley graph $P(q)$ is a $(q, \frac{q-1}{2}, \frac{q-5}{4}, \frac{q-1}{4})$-strongly regular graph with smallest eigenvalue $\frac{-1-r}{2}$. Since $C \cup \{0\}$ forms a maximum clique in $P_{q}$, it achieves the Hoffman bound \cite[Proposition 1.3.2]{BCN89}: given $u \notin C$, the number of neighbors of $u$ in $C \cup \{0\}$ is given by $\frac{r+1}{2}-1=\frac{r-1}{2}$. Finally, if $u \in \F_q^+$, one of the neighbors is $0$.
\end{proof}

The following proposition can be viewed as a strengthening of a result of Baker et.~al~\cite{B96}. A stronger statement can be found in \cite[Theorem 1.3]{GY25} for sufficiently large $q$. 

\begin{proposition}\label{prop:Galois}
Let $q=r^2$, where $r \equiv 1 \pmod 4$. If $u,v \in \F_{q} \setminus \F_r$ have the same $\F_r$-neighborhood in $P(q)$, then $v\in \{u,u^r\}$.        
\end{proposition}

We begin by proving the following lemma. For convenience, for each $u \in \F_q \setminus \F_r$, we use $L(u)$ to denote the set of neighbors of $u$ in $P(q)$ that lie in $\F_r$.

\begin{lemma}\label{lem:Galois}
Let $q=r^2$, where $r \equiv 1 \pmod 4$. If $u,v \in \F_{q} \setminus \F_r$ are distinct and have the same $\F_r$-neighborhood in $P(q)$, then $u+v \in \F_r$.   
\end{lemma}
\begin{proof}
First, assume that $u-v \in \F_r^*$. Let $x \in L(u)$.  Then $u-x=v-(x+(v-u)) \in \F_q^+$, and thus $x+(v-u) \in L(v)=L(u)$. Repeating the same argument, we must have $x+2(v-u), x+3(v-u), \cdots \in L(v).$ Therefore, for each $x \in \F_r$, we have $x \in L(u)$ if and only if $x+(v-u)\F_p \subset L(u)$. We conclude that $L(u)$ must be a union of additive $(v-u)\F_p$-cosets of $\F_r$. In particular, $|L(u)|$ is a multiple of $p$, that is, $p \mid \frac{r-1}{2}$, which is impossible.

Next, assume that $u-v \notin \F_r^*$. Then there exist $a \in \F_r$ and $t \in \F_r^* \setminus \{1\}$ such that $t(u-a)=v-a$. Indeed, we can identify $\F_q$ as an affine plane over $\F_r$, and the line passing through $u$ and $v$ intersects $\F_r$ at $a$.  Let $u'=u-a$ and $v'=v-a$, then $v'=tu'$. Note that $L(u')=L(u)-a=L(v)-a=L(v')$. Let $x \in L(u') \setminus \{0\}$. Then $x-u' \in \F_q^+$ and thus $tx-v'=t(x-u') \in \F_q^+$, which implies that $tx \in L(v')=L(u')$.  It follows that $t^jx \in L(u')$ for any positive integer $j$. Let $H$ be the subgroup of $\F_r^*$ generated by $t$, with $|H|=m$. Then the above argument shows that the $H$-coset containing $x$ is contained in $L(u')$. Thus, $L(u') \setminus \{0\}$ can be written a union of $H$-cosets in $\F_r^*$. In particular, $|L(u)|=|L(u')|=\frac{r-1}{2} \equiv 1 \pmod m$, that is, $m \mid \frac{r-3}{2}$. On the other hand, clearly $m \mid (r-1)$. It follows that  $m \mid 2$ and so $m=1$ or $m=2$. On the other hand, since $t \neq 1$, we have $m \geq 2$. Thus $m=2$, that is, $t=-1$ and we conclude $u+v=2a \in \F_r$, as claimed.   
\end{proof}

Now we use Lemma~\ref{lem:Galois} to prove Proposition~\ref{prop:Galois}.
\begin{proof}[Proof of Proposition~\ref{prop:Galois}]
Assume otherwise that $v \notin \{u, u^r\}$. Then Lemma~\ref{lem:Galois} implies that $u+v \in \F_r$. On the other hand, note that for each $x \in \F_r$, $u-x \in \F_q^+$ holds if and only if $u^r-x=(u-x)^r \in \F_q^+$. Thus, $L(u)=L(u^r)$, and similarly $L(v)=L(v^r)$. 

Then from $L(v^r)=L(v)=L(u)$ and $v^r \notin \{u,v\}$, Lemma~\ref{lem:Galois} implies that $u+v^r \in \F_r$ and $v^r+v \in \F_r$. We then conclude that $2u=(u+v)+(u+v^r)-(v^r+v) \in \F_r$, violating the assumption that $u \notin \F_r$.   
\end{proof}

We also need the following lemma concerning a geometric construction of a maximal clique or an independent set in the Paley graph $P(q)$, due to Goryainov et.~al \cite{GKSV18}.
\begin{lemma}[{\cite[Theorem 1]{GKSV18}}]\label{lem:oval}
Let $q=r^2$. Let $\Delta$ be an element in $\F_q^*$ with order $\frac{r+1}{2}$.  
\begin{enumerate}
    \item If $r \equiv 1 \pmod 4$, then $\{1, \Delta, \Delta^2, \ldots, \Delta^{\frac{r-1}{2}}\}$ is a maximal independent set in $P(q)$.
    \item If $r \equiv 3 \pmod 4$, then $\{1, \Delta, \Delta^2, \ldots, \Delta^{\frac{r-1}{2}}\} \cup \{0\}$ is a maximal clique in $P(q)$.
\end{enumerate}
\end{lemma}

\subsection{Outline of the proof}\label{SSOutlineSquare}

In this whole section, we assume $f:\F_q \to \F_q$ is a positivity preserver over $M_2(\F_q)$. Note that if $f$ is a positivity preserver, then for any $s \in \F_q^+$, the map $sf$ is also 
a positivity preserver. We therefore also assume without loss of generality that $f(1)=1$.

\begin{corollary}\label{cor:sqcoset}
The function $f$ maps a square coset to a square coset.
\end{corollary}
\begin{proof}
Let $C$ be a square coset. Then $C \subset \F_q^+$ and $C$ is a clique in $P(q)$. Lemma~\ref{lem:+} and Lemma~\ref{lem:sq} imply that $f(C) \subset \F_q^+$ and $f(C)$ is a clique in $P(q)$ of the same size. Corollary~\ref{cor:sqclique} then implies that $f(C)$ has to be a square coset.
\end{proof}

Since $f(1)=1$, $f$ maps the square coset $\F_r^*$ to itself.

\begin{corollary}\label{cor:f(0)=0}
We have $f(0)=0$.    
\end{corollary}
\begin{proof}
Let $x \in \F_r^*,$ and consider the matrix 
$$
A=
\begin{pmatrix} 
x&1\\
1&0
\end{pmatrix}.
$$    
Clearly, $A$ is positive definite. Under the map $f$, we have 
$$f(x)f(0)-f(1)^2=f(x)f(0)-1 \in \F_q^+$$
for all $x \in \F_r^*$. Using Corollary \ref{cor:sqcoset} and our $f(1)=1$ assumption, it follows that $f$ maps $\F_r^*$ bijectively into itself. This implies that
$xf(0)-1 \in \F_q^+$ for all $x \in \F_r^*$, and thus $f(0)-x \in \F_q^+$ for all $x \in \F_r^*$. 
In particular, the number of neighbors of $f(0)$ in $\F_r$ is at least $r-1$, and so we must have $f(0) \in \F_r$ by Lemma~\ref{lem:Hoffman}. Since $f(0)-x \in \F_q^+$ for all $x \in \F_r^*$, this forces $f(0)=0$.  
\end{proof}

\begin{proposition}\label{prop:power}
Let $\alpha \in \F_q^+$. There exist a positive integer $m=m(\alpha)$ such that $\gcd(m, r-1)=1$ and $f(\alpha x)=\beta x^m$ for all $x \in \F_r$, where $\beta=f(\alpha) \in \F_q^+$.
\end{proposition}

\begin{proof}
Let $\beta=f(\alpha)$ so that $f$ maps $\alpha \F_r^*$ to $\beta \F_r^*$. Define $\Tilde{f}(x)=f(\alpha x)/\beta$. Note that $\Tilde{f}(1)=1$ and $\Tilde{f}$ is bijective on $\F_r^*$.

Let $g$ be a primitive root of $\F_r$. Let $i$ be a positive integer.  
Consider the matrix 
$$
A=
\begin{pmatrix} 
a\alpha &g^i \alpha\\
g^i \alpha&b \alpha
\end{pmatrix}
$$    
with $a,b \in \F_r^*$. Note that $(ab-g^{2i})\alpha^2 \in \alpha^2 \F_r$, so if $ab\neq g^{2i}$, then $(ab-g^{2i})\alpha^2 \in \alpha^2\F_r^* \subset \F_q^+$ and the matrix $A$ is positive definite. Thus, if $ab \neq g^{2i}$, then $f[A]$ is also positive definite and thus 
$$
f(a\alpha) f(b \alpha) \neq f(g^i \alpha)^2,
$$
equivalently,
$$\Tilde{f}(a)\Tilde{f}(b) \neq \Tilde{f}(g^i)^2.
$$
Note that $\Tilde{f}$ is bijective on $\F_r^*$, thus, if $ab=g^{2i}$, we must have $\Tilde{f}(a)\Tilde{f}(b)=\Tilde{f}(g^i)^2$. We have thus proved that
\begin{equation}\label{eq:fq}
\Tilde{f}\bigg(\frac{g^{2i}}{a}\bigg)=\frac{\Tilde{f}(g^i)^2}{\Tilde{f}(a)}    
\end{equation}
for all $a \in \F_r^*$ and positive integers $i$. 

Next we use induction to prove $\Tilde{f}(g^j)=\Tilde{f}(g)^j$ for all $j$. 
\begin{itemize}
    \item Clearly the statement is true for $j=0,1$.
    \item By setting $a=1$ and $i=1$ in equation~\eqref{eq:fq}, we obtain that $\Tilde{f}(g^2)=\Tilde{f}(g)^2$. 
    \item If $j=2\ell+1$ is odd, set $i=\ell+1$ and $a=g$ in equation~\eqref{eq:fq}, we obtain that 
    $$\Tilde{f}(g^j)=\frac{\Tilde{f}(g^{\ell+1})^2}{\Tilde{f}(g)}=\Tilde{f}(g)^{2\ell+1}=\Tilde{f}(g)^j.
    $$
    \item If $j=2\ell$ is even, set $i=\ell$ and $a=1$ in equation~\eqref{eq:fq}, we obtain that 
    $\Tilde{f}(g^j)=\Tilde{f}(g^\ell)^{2}=\Tilde{f}(g)^j.$
\end{itemize}

Note that $\Tilde{f}(g)=h$ must be also a primitive root of $\F_r$. Say $h=g^m$; then $\gcd(m,r-1)=1$. For each $j$, we have $\Tilde{f}(g^j)=h^j=g^{mj}=(g^j)^m$, that is, $f(\alpha g^j)=\beta (g^j)^m$. This finishes the proof.
\end{proof}

The following proposition is key to determine the preservers in the $q=r^2$ case. Its proof is technical and is broken down into several propositions in Section \ref{SSInjectiveSquare}.

\begin{proposition}\label{prop:bij+}
The function $f$ maps different square cosets to different square cosets. Equivalently, $f$ is injective on $\F_q^+$. 
\end{proposition}

Combining Proposition~\ref{prop:main} and Proposition~\ref{prop:bij+}, we obtain the following theorem:
\begin{theorem}\label{thm:main}
If $f$ is a positivity preserver over $M_2(\F_q)$, where $q=p^{k} \equiv 1 \pmod 4$ is a square, then there are $a \in \F_q^+$ and $0\leq j\leq k-1$, such that
$f(x)=ax^{p^j}$ for all $x \in \F_q$.
\end{theorem}

\subsection{Proof of Proposition~\ref{prop:bij+}}\label{SSInjectiveSquare}
Let $g$ be a generator of $\F_q^*$. Then clearly the square cosets are given by $C_i=g^{2i} \F_r^*$ with $i=0,1, \ldots, \frac{r-1}{2}$. Identify $C_i$ with $C_{i+\frac{r+1}{2}}$. By Proposition~\ref{prop:power}, for each $i$, we can find $\beta_i \in \F_q^+$ and an integer $1\leq m_i<r-1$ with $\gcd(m_i,r-1)= 1$ such that we have $f(g^{2i}x)=\beta_ix^{m_i}$ for all $x \in \F_r^*$. Note that $m_{i+\frac{r+1}{2}}=m_i$ and $\beta_{i+\frac{r+1}{2}}=\beta_i g^{(r+1)m_i}$.

Let $i,j \geq 0$ be fixed and consider the square cosets $C_i, C_j, C_{2j-i}$. Suppose $x,z \in \F_r^*,$ and let
$$
A=
\begin{pmatrix} 
g^{2i}x&g^{2j}z\\
g^{2j}z& g^{4j-2i}y
\end{pmatrix},
$$    
where $y \in \F_r^*$. Note that $A$ is positive definite unless $y=z^2/x$. Thus, under the map $f$, we have $\det(f[A])=f(g^{2i}x)f(g^{4j-2i}y)-f(g^{2j}z)^2 \in \F_q^+$ for all $y \in \F_r^*\setminus \{z^2/x\}$.  We claim that $f(g^{2i}x)f(g^{4j-2i}\F_r^*)=f(g^{2j}z)^2\F_r^*$. Suppose otherwise, then $f(g^{2j}z)^2$ is not in the square coset $f(g^{2i}x)f(g^{4j-2i}\F_r^*)$. Lemma~\ref{lem:Hoffman} thus implies that the number of $y \in \F_r^*$ such that $f(g^{2i}x)f(g^{4j-2i}y)-f(g^{2j}z)^2 \in \F_q^+$ is $\frac{r-3}{2}<r-2,$ a contradiction. Therefore, we have $f(g^{2i}x)f(g^{4j-2i}\F_r^*)=f(g^{2j}z)^2\F_r^*$. Therefore, when $y=z^2/x$, we must have 

\begin{equation}\label{eq:inverse}
f(g^{2i}x)f(g^{4j-2i}z^2/x)=f(g^{2j}z)^2.
\end{equation}
Equation~\eqref{eq:inverse} implies that for all $x,z\in \F_r^*,$
\begin{equation}\label{eq:xz}
\beta_i x^{m_i} \beta_{2j-i} (z^2/x)^{m_{2j-i}}=\beta_{j}^2z^{2m_{j}}.
\end{equation}
Setting $z=1$ in equation~\eqref{eq:xz}, we obtain
$$
\beta_i \beta_{2j-i}  x^{m_i-m_{2j-i}}=\beta_{j}^2
$$
for all $x \in \F_r^*$, which implies that $m_i=m_{2j-i}$ and $\beta_i \beta_{2j-i}=\beta_{j}^2$. In particular, we have $m_0=m_2=\cdots$ and $m_1=m_3=\cdots$. Since $\beta_0=1$, inductively we have $\beta_i=\beta_1^i$. From now on, we set $\beta := \beta_1$ and $m:=m_0$.

Next we consider two cases, according to the value of $r$ modulo 4.

\subsubsection{The case $r \equiv 3 \pmod 4$}
In this case $\frac{r+1}{2}$ is even.

Let $t$ be the smallest positive integer such that $\beta^t \in \F_r^*$. Note that $\beta^{\frac{r+1}{2}}=\beta_{\frac{r+1}{2}}=g^{(r+1)m} \in \F_r^*$, so $t \mid \frac{r+1}{2}$. Also, note that $f(C_0), f(C_1), \ldots, f(C_{t-1})$ are different square cosets, and $f(C_0)=f(C_t)$. We need to show that $t=\frac{r+1}{2}$.

Assume that $t<\frac{r+1}{4}$ so that $\frac{r+1}{2t}>2$. Let $\beta^{t/m}:=(\beta^t)^{1/m}$. Note that this is well-defined since $\gcd(m,r-1) = 1$. Let $\Delta$ be an element in $\F_q^*$ with order $\frac{r+1}{2}$. Set $\widetilde{\Delta}=g^{2t} \beta^{-t/m}$. Note that for each $1\leq j<\frac{r+1}{2t}$, we have $\widetilde{\Delta}^{j}\neq 1$ since $g^{2tj} \notin \F_r$. On the other hand, since $\beta^{\frac{r+1}{2}}=g^{(r+1)m}$, we have $\beta^{t\frac{r+1}{2}}=g^{t(r+1)m}$. Since $\gcd(m, r-1)=1$, we have $g^{t(r+1)}=((\beta^{t})^{1/m})^{\frac{r+1}{2}}$. This shows that ${\widetilde{\Delta}}^{\frac{r+1}{2}}=1$ and the order of $\widetilde{\Delta}$ in $\F_q^*$ is at least $\frac{r+1}{2t}>2$. In particular, there exists an integer $1\leq \ell<\frac{r+1}{2}$ such that $\widetilde{\Delta}^2=\Delta^\ell$. Recall that for each $0 \leq i < \frac{r+1}{2t}$ with $i$ even, we have $m=m_i$ and thus $f(\widetilde{\Delta}^i)=f(g^{2ti}\beta^{-it/m})=\beta^{ti} (\beta^{-it/m})^m=1$. In particular, Lemma~\ref{lem:sq} implies that $1-\Delta^{\ell}=1-\widetilde{\Delta}^2 \in \F_q^-$, contradicting Lemma~\ref{lem:oval} (2).

Thus, in the following discussion, we can assume that $t \geq \frac{r+1}{4}$. Since $t \mid \frac{r+1}{2}$, we have $t=\frac{r+1}{4}$ or $t=\frac{r+1}{2}$. In the latter case, we are done. So we assume $t=\frac{r+1}{4}$. However, if $t=\frac{r+1}{4}$, then $\beta^{\frac{r+1}{2}}$ is a square in $\F_r^*$. On the other hand, since $g$ is a generator of $\F_q^*$, we have $g^{r+1}$ as a generator of $\F_r^*$. Since $\gcd(m, r-1)=1$, $g^{(r+1)m}$ remains a generator of $\F_r^*$. This contradicts $\beta^\frac{r+1}{2}=g^{(r+1)m}$.

\subsubsection{The case $r \equiv 1 \pmod 4$}

If $r \equiv 1 \pmod 4$, then $\frac{r+1}{2}$ is odd. Since $m_0=m_{\frac{r+1}{2}}$, this implies that $m_0=m_1=m_2=\cdots$.  Thus, we have $f(g^{2i}x)=\beta^ix^m$ for all $i$ and all $x \in \F_r^*$. In particular, it follows that $f$ is \emph{not} injective on $\F_q^+$ if and only if $f(C_i)=\F_r^*$ for a square coset $C_i$ other than $C_0$. Note that $f(C_i)=\F_r^*$ implies that $\beta^i \in \F_r^*$. Also, since $\gcd(m, r-1)=1$, we can find an integer $\ell$ such that $m\ell \equiv 1 \pmod {r-1}$, so that $\beta^{i/m}:=(\beta^i)^{1/m}=(\beta^i)^{\ell} \in \F_r^*$ is well-defined.

Recall that our goal is to show that $f$ is injective on $\F_{q}^{+}$. Suppose $f$ is not injective on $\F_{q}^{+}$, i.e., $f(C_i)=\F_r^*$ for some square coset $C_i$ other than $C_0$. Under this assumption, we obtain a contradiction via the next three propositions.

\begin{proposition}\label{prop:b}
Assume that $f(C_i)=\F_r^*$ for some square coset $C_i$ other than $C_0$. If $b \in \F_q^*$ such that $f(b)\not \in \F_r$, then $b^{r-1}=\beta^{i/m}g^{-2i}$. In particular, there are at most $r-1$ many $b's$ in $\F_q^*$ with $f(b) \not \in \F_r$.
\end{proposition}
\begin{proof}
For each $x \in \F_r^*$, by Lemma~\ref{lem:sq}, $x-b \in \F_q^+$ implies that $f(x)-f(b)\in \F_q^+$. Since $f(b)\notin \F_r$, it follows that $b \notin \F_r^*$ (as $f$ maps $\F_r^*$ to itself). We consider two cases:
\begin{itemize}
    \item If $b \in \F_q^+$, then we know that $f(b) \in \F_q^+$. Lemma~\ref{lem:Hoffman} implies that the number of neighbors of $b$ in $\F_r^*$ is $\frac{r-3}{2}$, and so is the number of neighbors of $f(b)$ in $\F_r^*$. Thus, for $x \in \F_r^*$, we have $x-b \in \F_q^+$ if and only if $f(x)-f(b)\in \F_q^+$.
    \item If $b \in \F_q^-$, then Lemma~\ref{lem:Hoffman} implies that the number of neighbors of $b$ in $\F_r^*$ is $\frac{r-1}{2}$. It follows that the number of neighbors of $f(b)$ in $\F_r^*$ is at least $\frac{r-1}{2}$. Thus, Lemma~\ref{lem:Hoffman} implies that $f(b)\in \F_q^-$ and the number of neighbors of $f(b)$ in $\F_r^*$ is exactly $\frac{r-1}{2}$. Therefore, for $x \in \F_r^*$, we have $x-b \in \F_q^+$ if and only if $f(x)-f(b)\in \F_q^+$.
\end{itemize}
In both cases, for $x \in \F_r^*$, we have $x-b \in \F_q^+$ if and only if $f(x)-f(b)\in \F_q^+$. By a similar argument, for $x \in \F_r^*$, we have $g^{2i}x-b \in \F_q^+$ if and only if $f(g^{2i}x)-f(b)\in \F_q^+$. 

Let $x \in \F_r^*,$ and recall that $f(x)=x^m$. We have
\begin{align*}
x-b \in \F_q^+ \iff x^m-f(b)\in \F_q^+ & \iff \beta^i(x/\beta^{i/m})^m-f(b)\in \F_q^+\\
\iff f(g^{2i}x/\beta^{i/m})-f(b) \in \F_q^+ & \iff g^{2i}x/\beta^{i/m}-b \in \F_q^+ \iff x-\beta^{i/m}g^{-2i}b \in \F_q^+.
\end{align*}
Therefore $b$ and $\beta^{i/m}g^{-2i}b$ share the same neighborhood in $\F_r$. Since $g^{2i}\notin \F_r$ and $\beta^{i/m} \in \F_r$, it follows $\beta^{i/m}g^{-2i}b \neq b$. Proposition~\ref{prop:Galois} implies that
$b^r=\beta^{i/m}g^{-2i}b$, i.e., $b^{r-1}=\beta^{i/m}g^{-2i}.$
\end{proof}

\begin{proposition}\label{prop:Fr}
We have $f(\F_q)=\F_r$.    
\end{proposition}
\begin{proof}
First we show that $f(\F_q^+)=\F_r^*$.  Recall that we have $f(C_i)=f(C_0)$ for a square coset $C_i$ different from $C_0$. If $f(C_1)\neq f(C_0)$, then we have $f(C_1)=f(C_{1+i}) \neq \F_r^*$, and thus there are at least $2(r-1)$ many $b$'s in $\F_q^*$ with $f(b)\notin \F_r$, contradicting Proposition~\ref{prop:b}. Thus, we must have $f(C_1)=f(C_0)$. It follows that $\beta \in \F_r^*$ and thus $f(g^{2j}x)=\beta^j x^m \in \F_r^*$ for all $j$ and $x \in \F_r^*$. In particular, $f(\F_q^+)=\F_r^*$.

Next we show that $f(\F_q)=\F_r$. Recall that by Corollary \ref{cor:f(0)=0}, we have $f(0) = 0$. Thus, $\F_r \subset f(\F_q)$. Suppose now that there exists $b \in \F_q^*$ with $f(b) \notin \F_r$. Since $f(\F_q^+)=\F_r^*$, Proposition~\ref{prop:b} applies to each square coset $C_i \neq C_0$, and thus $b^{r-1}=\beta^{1/m}g^{-2}=\beta^{2/m}g^{-4}$. Since $\beta \in \F_r^*$, it follows that $g^2 \in \F_r^*$, violating the assumption that $g$ is a generator of $\F_q^*$. Therefore, $f$ maps $\F_q$ to $\F_r$.   
\end{proof}

Recall from Corollary~\ref{cor:f(0)=0} that $f(0)=0$. To finish the proof, it suffices to show the following proposition, since it contradicts Lemma~\ref{lem:nonzero}.
\begin{proposition}
We have $f(\F_q^-)= \{0\}$.    
\end{proposition}
\begin{proof}
By Proposition~\ref{prop:Fr}, $f$ maps $\F_q$ to $\F_r$, and in particular $\beta \in \F_r^*$. Let $\Delta=g^2 \beta^{-1/m}\in \F_q^+$. Note that for each $1\leq j<\frac{r+1}{2}$, we have $\Delta^{j}\neq 1$ since $g^{2j} \notin \F_r$. On the other hand, since $\beta^{\frac{r+1}{2}}=\beta_{\frac{r+1}{2}}=g^{(r+1)m}$, and $\gcd(m, r-1)=1$, we have $g^{r+1}=(\beta^{1/m})^{\frac{r+1}{2}}$. This shows that $\Delta \in \F_q^*$ has order $\frac{r+1}{2}$. Lemma~\ref{lem:oval} (1) then implies that $I=\{1, \Delta, \Delta^2, \ldots, \Delta^{\frac{r-1}{2}}\} \subset \F_q^+$ forms a maximal independent set of size $\frac{r+1}{2}$ in $P(q)$.

Suppose there exists $w \in \F_q^-$ such that $f(w)=x \in \F_r^*$. Recall that $\gcd(m,r-1)$=1, in particular, $m$ is odd. Let $y=x^{(m+1)/2} \in \F_r^*$. Let $0 \leq i \leq \frac{r-1}{2}$ and consider the matrix
$$
A=
\begin{pmatrix} 
\Delta^ix&y^{1/m}\\
y^{1/m}& w
\end{pmatrix}.
$$    
Note that $f(\Delta^i x)=f(g^{2i} (\beta^{-i/m}x))=\beta^i (\beta^{-i/m}x)^m=x^m$ and $f(y^{1/m})=y=x^{(m+1)/2}$. Thus
$$
f[A]=
\begin{pmatrix} 
x^m&x^{(m+1)/2}\\
x^{(m+1)/2}& x
\end{pmatrix}
$$    
is not positive definite. Since $\Delta^ix \in \F_q^+$, this implies that $$\det(A)=\Delta^ixw-y^{2/m}=\Delta^i xw-x^{(m+1)/m}\notin \F_q^+,$$ and thus $w/x^{1/m}-\Delta^{-i} \notin \F_q^+$. Since $w \in \F_q^-$, we have $w/x^{1/m} \in \F_q^-$. It follows that $w/x^{1/m}-\Delta^i \in \F_q^-$ for all $0 \leq i \leq \frac{r-1}{2}$, which means that we can extend the maximal independent set $I \subset \F_q^+$ by adding a new element $w/x^{1/m} \in \F_q^-$, a contradiction. We have thus shown that $f(\F_q^-)=0$.     
\end{proof}

\section{Other approach: monomials via Lucas' theorem}\label{Smisc}
Throughout the section, we assume $q  = p^k\equiv 3 \pmod 4$. Recall that our proof of Theorem \ref{ThmB} relied on several lemmas and Weil's bound on character sums. Theorem \ref{ThmB} implies that the only power functions $f(x) = x^n$ that preserve positivity on $M_2(\F_q)$ are the field automorphisms $f(x) = x^{p^\ell}$ for some $0 \leq \ell \leq k-1$. We now provide an alternate proof for this fact using elementary number theory, which is of independent interest. The proof relies on Lucas' Theorem \cite{lucas1878theorie}, which we now recall.

For $a\in \{1,2,\ldots,q-1\}$, we denote the representation of $a$ in base $p$ by $a := (a_{k-1},\ldots,a_1,a_0)_p$, i.e., $a = a_{k-1}p^{k-1}+\ldots+a_1p+a_0$ where $0\leq a_i \leq p-1$ for each $0 \leq i \leq k-1$. The following classical result of Lucas provides an effective way to evaluate binomial coefficients modulo a prime. 
\begin{theorem}[Lucas \cite{lucas1878theorie}] \label{Lucas_Theorem}
    Let $a,b \in \{1,2,\ldots,q-1\}$. Then 
    $
        \binom{a}{b} \equiv \prod_{i=0}^{k-1}\binom{a_i}{b_i} \pmod{p},
    $
    where $a = (a_{k-1},\ldots,a_1,a_0)_p$ and $b = (b_{k-1},\ldots,b_1,b_0)_p$. 
\end{theorem}

\begin{lemma}\label{intermediate_lemma}
    Let $n\in \{1,2,\ldots,q-1\}$ such that $\gcd(n,q-1) = 1$ and $n\neq p^i$ for any $i=0,1,\ldots,k-1$. Then there exists a positive integer $r = r_{k-1}p^{k-1}+\ldots+r_1p+r_0$, where $0\leq r_i\leq \frac{p-1}{2}$ for all $0\leq i\leq k-1$, and such that if $s \in \{1,\dots, q-1\}$ and $s \equiv nr\pmod{q-1}$, then $\frac{q-1}{2}<s<q-1$. 
\end{lemma}
\begin{proof} Note that $\frac{q-1}{2} = \left(\frac{p-1}{2},\ldots,\frac{p-1}{2},\frac{p-1}{2}\right)_p$. Let $n = (n_{k-1},\ldots,n_1,n_0)_p$ and $t=\max\{n_i:0\leq i \leq k-1\}$. Denote by $j$ the largest integer such that $n_j = t$. Let us consider the following two cases.

\noindent{\bf Case 1: $t>1$.} Consider $r_j = \left\lfloor\frac{p-1}{2t} \right\rfloor + 1$ and $r = r_jp^{k-1-j}$. Then we obtain 
\begin{align*}
nr &= \left(\sum_{i=0}^{k-1}n_ip^i\right)r_jp^{k-1-j}\equiv \sum_{i=0}^{j}n_ir_jp^{k-1-(j-i)}+\sum_{i=j+1}^{k-1}n_ir_jp^{i-(j+1)} \\
& \equiv \sum_{\ell=0}^{k-j-2}n_{\ell+j+1}r_jp^{\ell}+\sum_{i=0}^{j}n_ir_jp^{k-1-(j-i)}  \pmod {q-1}.
\end{align*}
Let $s=(n_jr_j,n_{j-1}r_j,\ldots,n_0r_j,n_{k-1}r_j,n_{k-2}r_j\ldots,n_{j+1}r_j)_p$. Then we have $s \in \{1,\dots, q-1\}$ and $s \equiv nr\pmod{q-1}$. Note that $1\leq r_j \leq \frac{p-1}{2}$, $n_jr_j > \frac{p-1}{2}$, and $0\leq n_ir_j \leq p-1$ for all $i = 0,1,\ldots,k-1$. Also, $s \ne q-1$ since $\gcd(n,q-1) = 1$. It follows that $q-1 > s>\frac{q-1}{2}$.

\noindent{\bf Case 2: $t=1$.} Now assume $t=1$. Then $n_{i} \in \{0,1\}$ for all $i=0,1,\ldots,k-1$. Since $n\neq p^i$ for any $i=0,1,\ldots,k-1$, there exist two distinct integers, say $j$ and $\ell$, such that $n_{j} = n_{\ell} = 1$. Let $r_j = r_{\ell} = \frac{p-1}{2}$ and let $r = r_jp^{k-1-j} + r_{\ell} p^{k-1-\ell}$. By a similar calculation as in the previous case, if $s = (s_{k-1},\ldots,s_{1},s_0)_p$ with $s \equiv nr \pmod{q-1}$, then $s_{k-1} = p-1$ and $s_{i} \in \{0,\frac{p-1}{2},p-1\}$ for all $i=0,1,\ldots,k-1$. Since $\gcd(n,q-1) = 1$, $s \ne q-1$ and it follows that $q-1 > s > \frac{q-1}{2}$. 
\end{proof}

Let $g(x) = \sum_{i=0}^{m}a_ix^i$ be a polynomial of degree $m$ in $\F_q[x]$. Suppose $r(x)$ is the remainder obtained from $g(x)$ when dividing it by $x^q-x$. Then $g$ has degree at most $q-1$ and $g(x) \equiv r(x)\pmod{x^q-x}$. We may avoid long division when dividing a polynomial by $x^q-x$ since $x^q = x$ for all $x\in \F_q$. More precisely, $r(x) = a_0+\sum_{i=1}^{m}a_ix^{m\pmod{q-1}}$ with the convention that $m \pmod{q-1}$ is the unique integer $m'$ such that $1\leq m' \leq q-1$ and $m' \equiv m \pmod {q-1}$. 

\begin{corollary} \label{Key_Lemma}
      Let $n\in \{1,2,\ldots,q-1\}$ such that $\gcd(n,q-1) = 1$. Define $g(x) = (x^n-1)^{\frac{q-1}{2}}$ and $h(x) = (x-1)^{\frac{q-1}{2}}$. Then $g(c) =h(c)$ for all $c\in \F_q$ if and only if $n = p^i$ for some $i\in \{0,1,\ldots,k-1\}$.  
\end{corollary}
\begin{proof}
    Suppose $n = p^i$ for some $i\in \{0,1,\ldots,k-1\}$. Then for any $c\in \F_q$ we have
    \begin{align*}
        g(c) = (c^n-1)^{\frac{q-1}{2}} = (c^{p^i}-1)^{\frac{q-1}{2}} = (c-1)^{p^i \cdot \frac{q-1}{2}} = h(c)^{p^{i}}. 
    \end{align*}
    So $g(c) = h(c)$ for all $c\in \F_q$ since $g(c),h(c) \in \{-1,0,1\}$ and $p$ is odd. Conversely, suppose $n \neq p^i$ for any $i\in \{0,1,\ldots,k-1\}$. Note that $\deg(h(x)) = \frac{q-1}{2}$. On the other hand, we have
    \begin{align*}
        g(x) &= (x^n-1)^{\frac{q-1}{2}}
        = \sum_{r=0}^{\frac{q-1}{2}}(-1)^{\frac{q-1}{2}-r}\binom{\frac{q-1}{2}}{r}x^{nr}\\
        &\equiv -1 + \sum_{r=1}^{\frac{q-1}{2}}\left\{(-1)^{\frac{q-1}{2}-r}\binom{\frac{q-1}{2}}{r} \pmod{p}\right\}x^{nr\pmod{q-1}} \pmod{x^q-x}.
    \end{align*}

By Lucas's theorem (Theorem \ref{Lucas_Theorem}) and  Lemma \ref{intermediate_lemma} we must have $\deg\left(g(x)\pmod{x^q-x}\right) > \frac{q-1}{2}$. Thus $g(x) \not \equiv h(x) \pmod{x^q-1}$. The result now follows from Lemma \ref{Lidl_Lemma}. 
\end{proof}

We now directly examine the properties of power functions that preserve positivity on $M_2(\F_q)$. 

\begin{lemma}\label{even_monomials}
    Let $f(x) = x^n$ for some $n \in \{1,2,\ldots,q-1\}$. If $n$ is even, then $f$ does not preserve positivity on $M_2(\F_q)$. 
\end{lemma}

\begin{proof} Suppose $n$ is even and $f(x)=x^n$ preserves positivity on $M_2(\F_q)$. Then Lemma~\ref{L:Charq} implies that $f(x)$ must be bijective on $\F_q^+$ onto itself and $f(0) = 0$. Since $f(x)$ is even, $f$ maps $\F_q^-$ bijectively onto $\F_q^+$, and thus $f$ restricted to $\F_q^*$ is a $2$-to-$1$ map. It follows that $\{f(z+1): z \in \F_q^+\} \subset \F_q^+$ has size at least $\lceil|\F_q^+|/2 \rceil=\frac{q+1}{4}$. From $|\F_q^+ \cap (-1+\F_q^+)|= \frac{q-3}{4}$ (Lemma~\ref{Ldoublesquare}), there exists $z \in \F_q^+$ such that $f(z+1)-1 \notin \F_q^+$. For such $z$, the matrix $A = \begin{pmatrix}
    1 & 1\\
    1 & z+1
\end{pmatrix}$ is positive definite but $f[A] = \begin{pmatrix}
    1 & 1\\
    1 & f(z+1)
\end{pmatrix}$ is not, a contradiction.\end{proof}

\begin{lemma}\label{imp_cri_for_mono}
    Let $f(x) = x^n$ for some $n \in \{1,2,\ldots,q-1\}$. If $f$ preserves positivity on $M_2(\F_q)$, then $\gcd(n,q-1) = 1$ and $\eta(a-1)=\eta(a^n-1)$ for all $a \in \F_q$.
\end{lemma}
\begin{proof} By Lemma \ref{even_monomials}, $n$ is odd. Lemma \ref{L:Charq} then implies that $f(x) = x^n$ is a bijective map, and thus we must have $\gcd(n,q-1) = 1$ by Theorem \ref{Tperm}(2). For the sake of contradiction, assume that there is $a \in \F_q$ such that $\eta(a-1)\neq \eta(a^n-1)$. Clearly, $a \neq 0$. We consider the following three cases:
\begin{description}
\item[Case 1] $\eta(a-1)=1$ and $\eta(a^n-1)=-1$. Consider the matrix $A = \begin{pmatrix}
1 & 1\\
1 & a
\end{pmatrix}$. 
Then $A$ is positive definite, but $f[A] = \begin{pmatrix}
1 & 1\\
1 & a^n
\end{pmatrix}$
is not, a contradiction. 
\item[Case 2] $a \in \F_q^+$, $\eta(a-1)=-1$, and $\eta(a^n-1)=1$. Then $\sqrt{a}$ exists and consider the matrix $A = \begin{pmatrix}
1 & \sqrt{a}\\
\sqrt{a} & 1
\end{pmatrix}$. Then $A$ is positive definite, but $f[A] = \begin{pmatrix}
1 & (\sqrt{a})^n\\
(\sqrt{a})^n & 1
\end{pmatrix}$ is not. 

\item[Case 3] $a \in \F_q^-$, $\eta(a-1)=-1$, and $\eta(a^n-1)=1$. In this case $\sqrt{-a}$ is well-defined. Clearly $a \ne -1$, and we now consider $a+1 \in \F_q^*$. Suppose $a+1 \in \F_q^+$. Consider the matrix $A = \begin{pmatrix}
1 & \sqrt{-a}\\
\sqrt{-a} & 1
\end{pmatrix}$. 
Then $A$ is positive definite and therefore so is $f[A] = \begin{pmatrix}
1 & (\sqrt{-a})^n\\
(\sqrt{-a})^n & 1
\end{pmatrix}$.
Thus, $\det f[A] = a^n+1 \in \F_q^+$ is positive. Now, $a^2-1 = (a-1)(a+1) \in \F_q^-$ and $(a^2)^n-1 = (a^n-1)(a^n+1) \in \F_q^+$. Taking $b=a^2$, we have $b \in \F_q^+$, $b-1 \in \F_q^-$ and $b^n -1 \in \F_q^+$. By Case 2 above applied to $b$, we conclude that $f$ does not preserve positivity. Finally, suppose $a+1 \in \F_q^-$. Consider the matrix $A = \begin{pmatrix}
\sqrt{-a} & 1\\
1 & \sqrt{-a}
\end{pmatrix}$. 
Then $A$ is positive definite and so is $f[A] = \begin{pmatrix}
(\sqrt{-a})^n & 1\\
1& (\sqrt{-a})^n
\end{pmatrix}$. Thus, $\det f[A] = -(a^n+1) \in \F_q^+$. Hence, $a^2-1 = (a-1)(a+1)\in \F_q^+$ and $(a^2)^n-1 = (a^n-1)(a^n+1) \in \F_q^-$. Applying Case 1 above to $b=a^2$, we conclude that $f$ does not preserve positivity on $M_2(\F_q)$. \qedhere
\end{description}
\end{proof}

Finally, we obtain the desired result.
\begin{theorem}\label{Odd_char_odd_monomial}
  Let $n \in \{1,2,\ldots,q-1\}$. Then $f(x) = x^n$ preserves positivity on $M_2(\F_q)$ if and only if $n = p^i$ for some $i\in \{0,1,\ldots,k-1\}$. 
\end{theorem}
\begin{proof} Suppose $n= p^i$ for some $i\in \{0,1,\ldots,k-1\}$. Then by Proposition \ref{PFrobenius}, $f(x) = x^n$ preserves positivity on $M_2(\F_q)$. Conversely, suppose $f(x) = x^n$ preserves positivity on $M_2(\F_q)$. Lemma \ref{imp_cri_for_mono} implies that $\gcd(n,q-1) = 1$ and $\eta(a-1)=\eta(a^n-1)$ for all $a \in \F_q$. Now, consider the following two functions
\begin{align*}
    g(x) = (x^n-1)^{\frac{q-1}{2}} = \eta(x^n-1), \quad 
    h(x) = (x-1)^{\frac{q-1}{2}} = \eta(x-1).
\end{align*}
We have $g(c) = h(c)$ for all $c \in \F_q$. Corollary~\ref{Key_Lemma} then implies the desired conclusion.
\end{proof}

\section{Conclusion}\label{Sconc}
The astute reader will have noticed that one case was not addressed in the paper: the characterization of entrywise preservers on $M_2(\F_q)$ when $q \equiv 1 \pmod 4$ and $q$ is not a square. When $q=r^2$, our proof took advantage of the better understood structure of the cliques in the Paley graph $P(q)$. While the authors were able to gather evidence that the analog of Theorem \ref{ThmB} should hold when $q$ is not a square, our techniques did not allow us to resolve it. We note, however, that the sufficient conditions obtained in Section \ref{SSSufficient} and \ref{SSApplications} still apply to this case. Resolving the general case will be the object of future work:

\begin{question}
If $f$ preserves positivity on $M_2(\F_q)$ where $q \equiv 1 \pmod 4$ is not a square, does $f$ have to be injective?    
\end{question}

Finally, recall that Schoenberg's theorem (Theorem \ref{Tschoenberg}) addresses preservers of both positive semidefiniteness and of positive definiteness for matrices with real entries. While our current work focuses on preservers of positive definiteness over finite fields, it would be interesting to investigate appropriate notions of positive semidefiniteness over $\F_q$, and to work out the associated positive semidefiniteness preservers.

\section*{Acknowledgments}
The authors would like to acknowledge the American Institute of Mathematics (CalTech) for their hospitality and stimulating environment during a workshop on Theory and Applications of Total Positivity in July 2023 where the first three authors met and initial discussions occurred. The authors would also like to thank Apoorva Khare, Felix Lazebnik, and an anonymous referee for their comments on the paper.

D.G.~was partially supported by a Simons Foundation collaboration grant for mathematicians and by NSF grant \#2350067. H.G.~and P.K.V.~acknowledge support from PIMS (Pacific Institute for the Mathematical Sciences) Postdoctoral Fellowships. P.K.V. was additionally supported by a SwarnaJayanti Fellowship from DST and SERB (Govt.~of India), and is moreover thankful to the SPARC travel support (Scheme for Promotion of Academic and Research Collaboration, MHRD, Govt.~of India; PI: Tirthankar Bhattacharyya, Indian Institute of Science), and the University of Plymouth (UK) for hosting his visit during part of the research. C.H.Y.~was partially supported by an NSERC fellowship.

\bibliographystyle{plain}
\bibliography{biblio}

\begin{thebibliography}{10}

\bibitem{AY22}
Shamil Asgarli and Chi~Hoi Yip.
\newblock Van {L}int--{M}ac{W}illiams' conjecture and maximum cliques in
  {C}ayley graphs over finite fields.
\newblock {\em J. Combin. Theory Ser. A}, 192:Paper No. 105667, 23, 2022.

\bibitem{B96}
Ronald~D. Baker, Gary~L. Ebert, Joe Hemmeter, and Andrew Woldar.
\newblock Maximal cliques in the {P}aley graph of square order.
\newblock {\em J. Statist. Plann. Inference}, 56(1):33--38, 1996.

\bibitem{BGKP-fixeddim}
Alexander Belton, Dominique Guillot, Apoorva Khare, and Mihai Putinar.
\newblock Matrix positivity preservers in fixed dimension. {I}.
\newblock {\em Adv. Math.}, 298:325--368, 2016.

\bibitem{BGKP-survey-part-1}
Alexander Belton, Dominique Guillot, Apoorva Khare, and Mihai Putinar.
\newblock A panorama of positivity. {I}: {D}imension free.
\newblock In Alexandru Aleman, H\r{a}kan Hedenmalm, Dmitry Khavinson, and Mihai
  Putinar, editors, {\em Analysis of Operators on Function Spaces, The Serguei
  Shimorin memorial volume}, Trends in Mathematics, pages 117--165. Birkhauser,
  Basel, 2019.

\bibitem{BGKP-survey-part-2}
Alexander Belton, Dominique Guillot, Apoorva Khare, and Mihai Putinar.
\newblock A panorama of positivity. {II}: {F}ixed dimension.
\newblock In J.~Mashreghi G.~Dales, D.~Khavinson, editor, {\em Complex Analysis
  and Spectral Theory: Proceedings of the CRM Workshop held at Laval
  University, QC, May 21–25, 2018}, CRM Proceedings, AMS Contemporary
  Mathematics 743,, pages 109--150. American Mathematical Society, Providence,
  RI, 2020.

\bibitem{BGKP-hankel}
Alexander Belton, Dominique Guillot, Apoorva Khare, and Mihai Putinar.
\newblock Moment-sequence transforms.
\newblock {\em J. Eur. Math. Soc.}, 24(9):3109--3160, 2022.

\bibitem{belton2023negativity}
Alexander Belton, Dominique Guillot, Apoorva Khare, and Mihai Putinar.
\newblock Negativity-preserving transforms of tuples of symmetric matrices.
\newblock {\em arXiv preprint arXiv:2310.18041}, 2023.

\bibitem{BGKP-tn}
Alexander Belton, Dominique Guillot, Apoorva Khare, and Mihai Putinar.
\newblock Totally positive kernels, {P}{\'o}lya frequency functions, and their
  transforms.
\newblock {\em J. d'Analyse. Math.}, 150(1):83--158, 2023.

\bibitem{Blo84}
Aart Blokhuis.
\newblock On subsets of {${\rm GF}(q^2)$} with square differences.
\newblock {\em Nederl. Akad. Wetensch. Indag. Math.}, 46(4):369--372, 1984.

\bibitem{BCN89}
Andries~E. Brouwer, Arjeh~M. Cohen, and Arnold Neumaier.
\newblock {\em Distance-{R}egular {G}raphs}, volume~18 of {\em Ergebnisse der
  Mathematik und ihrer Grenzgebiete (3) [Results in Mathematics and Related
  Areas (3)]}.
\newblock Springer-Verlag, Berlin, 1989.

\bibitem{brouwer2011spectra}
Andries~E. Brouwer and Willem~H. Haemers.
\newblock {\em Spectra of {Graphs}}.
\newblock Springer Science \& Business Media, 2011.

\bibitem{BM22}
Andries~E. Brouwer and William~J. Martin.
\newblock Triple intersection numbers for the {P}aley graphs.
\newblock {\em Finite Fields Appl.}, 80:Paper No. 102010, 4, 2022.

\bibitem{carlitz1960theorem}
Leonard Carlitz.
\newblock A theorem on permutations in a finite field.
\newblock {\em Proc. Amer. Math. Soc.}, 11(3):456--459, 1960.

\bibitem{cooper2022positive}
Joshua Cooper, Erin Hanna, and Hays Whitlatch.
\newblock Positive-definite matrices over finite fields.
\newblock {\em Rocky Mountain J. Math.}, 54(2):423--438, 2024.

\bibitem{EKR61}
Paul Erd\H{o}s, Chao Ko, and Richard Rado.
\newblock Intersection theorems for systems of finite sets.
\newblock {\em Quart. J. Math. Oxford Ser. (2)}, 12:313--320, 1961.

\bibitem{fitzgerald1977fractional}
Carl~H. FitzGerald and Roger~A. Horn.
\newblock On fractional hadamard powers of positive definite matrices.
\newblock {\em J. Math. Anal. Appl.}, 61(3):633--642, 1977.

\bibitem{fitzgerald1995functions}
Carl~H. FitzGerald, Charles~A. Micchelli, and Allan Pinkus.
\newblock Functions that preserve families of positive semidefinite matrices.
\newblock {\em Linear Algebra Appl.}, 221:83--102, 1995.

\bibitem{godsil2016erdos}
Chris Godsil and Karen Meagher.
\newblock {\em {Erd\H{o}s}-{K}o-{R}ado {T}heorems: {A}lgebraic {A}pproaches},
  volume 149 of {\em Cambridge Studies in Advanced Mathematics}.
\newblock Cambridge University Press, Cambridge, 2016.

\bibitem{GKSV18}
Sergey Goryainov, Vladislav~V. Kabanov, Leonid Shalaginov, and Alexandr
  Valyuzhenich.
\newblock On eigenfunctions and maximal cliques of {P}aley graphs of square
  order.
\newblock {\em Finite Fields Appl.}, 52:361--369, 2018.

\bibitem{guillot2015complete}
Dominique Guillot, Apoorva Khare, and Bala Rajaratnam.
\newblock Complete characterization of {H}adamard powers preserving loewner
  positivity, monotonicity, and convexity.
\newblock {\em J. Math. Anal. Appl.}, 425(1):489--507, 2015.

\bibitem{GKR-critG}
Dominique Guillot, Apoorva Khare, and Bala Rajaratnam.
\newblock Critical exponents of graphs.
\newblock {\em J. Combin. Theory Ser. A}, 139:30--58, 2016.

\bibitem{GKR-sparse}
Dominique Guillot, Apoorva Khare, and Bala Rajaratnam.
\newblock Preserving positivity for matrices with sparsity constraints.
\newblock {\em Trans. Amer. Math. Soc.}, 368(12):8929--8953, 2016.

\bibitem{GKR-lowrank}
Dominique Guillot, Apoorva Khare, and Bala Rajaratnam.
\newblock Preserving positivity for rank-constrained matrices.
\newblock {\em Trans. Amer. Math. Soc.}, 369(9):6105--6145, 2017.

\bibitem{guillot2012retaining}
Dominique Guillot and Bala Rajaratnam.
\newblock Retaining positive definiteness in thresholded matrices.
\newblock {\em Linear Algebra Appl.}, 436(11):4143--4160, 2012.

\bibitem{guillot2015functions}
Dominique Guillot and Bala Rajaratnam.
\newblock Functions preserving positive definiteness for sparse matrices.
\newblock {\em Trans. Amer. Math. Soc.}, 367(1):627--649, 2015.

\bibitem{guterman2000some}
Alexander Guterman, Chi-Kwong Li, and Peter {\v{S}}emrl.
\newblock Some general techniques on linear preserver problems.
\newblock {\em Linear algebra and its applications}, 315(1-3):61--81, 2000.

\bibitem{HP}
Brandon Hanson and Giorgis Petridis.
\newblock Refined estimates concerning sumsets contained in the roots of unity.
\newblock {\em Proc. Lond. Math. Soc. (3)}, 122(3):353--358, 2021.

\bibitem{hiai2009monotonicity}
Fumio Hiai.
\newblock Monotonicity for entrywise functions of matrices.
\newblock {\em Linear Algebra Appl.}, 431(8):1125--1146, 2009.

\bibitem{horn1969theory}
Roger~A. Horn.
\newblock The theory of infinitely divisible matrices and kernels.
\newblock {\em Trans. Amer. Math. Soc.}, 136:269--286, 1969.

\bibitem{horn2012matrix}
Roger~A. Horn and Charles~R. Johnson.
\newblock {\em Matrix {Analysis}}.
\newblock Cambridge university press, 2012.

\bibitem{jones2020paley}
Gareth~A. Jones.
\newblock Paley and the {P}aley graphs.
\newblock In {\em Isomorphisms, Symmetry and Computations in Algebraic Graph
  Theory: Pilsen, Czech Republic, October 3--7, 2016}, pages 155--183.
  Springer, 2020.

\bibitem{khare2022matrix}
Apoorva Khare.
\newblock {\em Matrix {Analysis} and {Entrywise} {Positivity} {Preservers}}.
\newblock London Mathematical Society Lecture Note Series, Cambridge University
  Press, 2022.

\bibitem{Khare-Tao}
Apoorva Khare and Terence Tao.
\newblock On the sign patterns of entrywise positivity preservers in fixed
  dimension.
\newblock {\em Amer. J. Math.}, 143(6):1863--1929, 2021.

\bibitem{koblitz1994course}
Neal Koblitz.
\newblock {\em A {Course} in {Number} {Theory} and {Cryptography}}, volume 114.
\newblock Springer Science \& Business Media, 1994.

\bibitem{li2001linear}
Chi-Kwong Li and Stephen Pierce.
\newblock Linear preserver problems.
\newblock {\em The American Mathematical Monthly}, 108(7):591--605, 2001.

\bibitem{lidl1997finite}
Rudolf Lidl and Harald Niederreiter.
\newblock {\em Finite {Fields}}.
\newblock Encyclopedia of Mathematics and its Applications, Cambridge
  university press, second edition, 1997.

\bibitem{lucas1878theorie}
Edouard Lucas.
\newblock Th{\'e}orie des fonctions num{\'e}riques simplement p{\'e}riodiques.
\newblock {\em Amer. J. Math.}, 1(2):289--321, 1878.

\bibitem{GY25}
Greg Martin and Chi~Hoi {Yip}.
\newblock {Distribution of power residues over shifted subfields and maximal
  cliques in generalized Paley graphs}.
\newblock {\em Proc. Amer. Math. Soc.}, 153(1):109--124, 2025.

\bibitem{muzychuk2005solution}
Mikhail Muzychuk and Istv\'an Kov\'acs.
\newblock A solution of a problem of {A}. {E}. {B}rouwer.
\newblock {\em Des. Codes Cryptogr.}, 34(2-3):249--264, 2005.

\bibitem{orel2016preserver}
Marko Orel.
\newblock Preserver problems over finite fields.
\newblock {\em International Journal of Mathematics, Game Theory, and Algebra},
  25(3), 2016.

\bibitem{Rudin-Duke59}
Walter Rudin.
\newblock Positive definite sequences and absolutely monotonic functions.
\newblock {\em Duke Math. J.}, 26:617--622, 1959.

\bibitem{Schoenberg-Duke42}
Isaac~J. Schoenberg.
\newblock Positive definite functions on spheres.
\newblock {\em Duke Math. J.}, 9:96--108, 1942.

\bibitem{Schur1911}
Issai {S}chur.
\newblock {B}emerkungen zur {T}heorie der beschr{\"a}nkten {B}ilinearformen mit
  unendlich vielen {V}er{\"a}nderlichen.
\newblock {\em J. reine angew. Math.}, 140:1--28, 1911.

\bibitem{vishwakarma2023positivity}
Prateek~Kumar Vishwakarma.
\newblock Positivity preservers forbidden to operate on diagonal blocks.
\newblock {\em Trans. Amer. Math. Soc.}, 376:5261--5279, 2023.

\bibitem{Y22}
Chi~Hoi Yip.
\newblock On the clique number of {P}aley graphs of prime power order.
\newblock {\em Finite Fields Appl.}, 77:Paper No. 101930, 16, 2022.

\end{thebibliography}
\end{document}